\documentclass[11pt]{article}

\usepackage{amsfonts,amssymb}

\usepackage{amsthm}

\usepackage{amsmath}

\usepackage{amscd}
\usepackage{tabularx}
\usepackage{arydshln}
\usepackage{verbatim}
\usepackage{color}
\usepackage{a4wide}
 
\usepackage[all]{xy}

\theoremstyle{plain}

\newtheorem{theorem}{Theorem}[section]
\newtheorem{lemma}[theorem]{Lemma}

\newtheorem{corollary}[theorem]{Corollary}
\newtheorem{Counter-example}[theorem]{Counter-example}
\newtheorem{remark}[theorem]{Remark}
\newtheorem{example}[theorem]{Example}

\theoremstyle{definition}

\newtheorem{definition}[theorem]{Definition}

\theoremstyle{remark}

\usepackage{amssymb}

\def\R{\mathbb R}

\def\Z{\mathbb Z}

\def\S{\mathbb S}
\def\L{\mathbb L}

\def\Z{\mathbb Z}
\def\E{\mathbb E}

\begin{document}

\title{Lightlike manifolds and Cartan Geometries}

\author{
	\emph{Francisco J. Palomo}
	\thanks{The author is partially supported by the Spanish MINECO and ERDF project
MTM2016-78807-C2-2-P
and by the Junta Andaluc\'{\i}a  and ERDF I+D+I project A-FQM-494-UGR18.} \\
	 Departamento de Matem\'{a}tica Aplicada \\
Universidad de M\'{a}laga,  29071-M\'{a}laga (Spain)	\\
Email address: \emph{ fpalomo@uma.es} 
\date{\empty}
}

\maketitle


\begin{abstract}
Lightlike Cartan geometries are introduced as Cartan geometries modelled on the future lightlike cone in Lorentz-Minkowski spacetime. Then, we provide an approach to the study of lightlike manifolds from this point of view. It is stated that every lightlike Cartan geometry on a manifold $N$ provides a lightlike metric $h$ with radical distribution globally spanned by a vector field $Z$. For lightlike hypersurfaces of a Lorentz manifold, we give the condition that characterizes when the pull-back of the Levi-Civita connection form of the ambient manifold is a lightlike Cartan connection on such hypersurface.
In the particular case that a lightlike hypersurface is properly totally umbilical, this construction essentially returns  the original lightlike metric. From the intrinsic point of view, starting from a given lightlike manifold $(N,h)$, we show a method to construct  a family of ambient Lorentzian manifolds that realize $(N,h)$ as a hypersurface. This method is inspired on the Feffermann-Graham ambient metric construction in conformal geometry and provides a lightlike Cartan geometry on the original manifold when $(N,h)$ is generic.

 \end{abstract}




\section{Introduction}

A lightlike metric on a manifold $N$ is a symmetric $(0,2)$-tensor $h$, which is positive but non-definite, and whose radical
$\mathrm{Rad}(h):=\{v\in TN: h(v,.)=0\}$ defines a $1$-dimensional distribution on $N$ (see Definition \ref{52}). 
Thus, the radical distribution $\mathrm{Rad}(h)$ is the smallest possible one that assures degeneracy of $h$ at every point.
The lightlike manifolds  $(N,h)$ naturally appear as hypersurfaces of Lorentzian manifolds and this was our original motivation. The terminology \emph{lightlike manifolds} stems from General Relativity where lightlike hypersurfaces are models of various types of horizons. Roughly speaking, the horizon of a set $A$ marks the limit of the region of such spacetime controlled by a set $A$ \cite[Chap. 14]{One83}.
Also in General Relativity, the existence of smooth closed achronal totally geodesic lightlike hypersurfaces (the Null Splitting Theorem) \cite[Theor. IV.1 ]{Galloway} has important consequences in order to obtain rigidity results (see for instance  \cite[Theor. IV.3 ]{Galloway}). Lightlike hypersurfaces also appear as characteristic submanifolds for the wave equation on a Lorentzian manifold.

The existence of the radical distribution $\mathrm{Rad}(h)$ in a lightlike manifold $(N,h)$ prevents from determining a linear connection on $N$, which is a key difference with the case of semi-Riemannian manifolds. Thus, there is no distinguished covariant derivative on lightlike manifolds, in general. Moreover, whenever the lightlike manifold is realized as a hypersurface in a Lorentzian manifold, its normal fiber bundle intersects with the tangent bundle of the hypersurface. There is no decomposition of the tangent bundle of the ambient Lorentzian manifold along the immersion and hence, the classical theory of non-degenerate submanifolds is not applicable in this case.

Let us recall several approaches that have been developed along the years to deal with such difficulties in the study of lightlike manifolds, with no claim to be exhaustive. As far as we know, one of the pioneering works is \cite{Penrose80} (this article originally appeared in 1963). In this paper, R. Penrose considered the initial value problem in General Relativity when the initial hypersurface is lightlike, rather than spacelike. Several points of view with a more geometric flavour were developed in the seventies (see for instance \cite{Rosca72} and \cite{Bonnor}).
A linear connection for lightlike hypersurfaces is introduced in \cite{Katsuno1,Katsuno2} by Katsuno. The analysis in these papers is developed by means of  local coordinates with the restriction that ambient Lorentzian manifold is assumed to be  $4$-dimensional.

The notion of lightlike hypersurface with base on a spacelike surface appears in \cite{Koss} where a regularity condition is assumed.
In words of Kossowski, this condition {\it prohibits the lightlike hypersurface from being too flat}  (see Example \ref{10123}). In
1996, Bejancu and Duggal introduced a general geometric technique to deal with lightlike
hypersurfaces \cite {DB96}, consisting in the use of a  non-degenerate
screen distribution, which has since been used extensively.
The geometry of {\it degenerate submanifolds} of
semi-Riemannian manifolds of arbitrary signature is also thoroughly discussed in \cite{Kupeli}, where  additional geometric structures in the ambient manifold are considered. 

Several of the above mentioned techniques rest on a certain {\it normalization} of the lightlike hypersurface $N$, in the sense that  they assign either a lightlike vector field  that does not belong to the tangent hyperplanes of $N$, or a spacelike screen distribution on $N$. These two normalization methods of a lightlike hypersurface are equivalent. In this direction,  methods of construction of such normalization are presented in  \cite{AkGold}, the Cartan's method for dealing with local equivalence problems has been applied to $3$-dimensional lightlike manifolds in \cite{NuRo}, and Lorentzian manifolds $(M,g)$ which admit recurrent lightlike vector fields $Z$ have been studied  in \cite{Leistner}. In this latter particular case, $M$ is foliated by lightlike hypersurfaces (see Example \ref{1701201}).

To round off this overview, we mention two facts that will be relevant for this paper: on the one hand, the rigidity of the future lightcone in Minkowski space as obtained in \cite{BFZ}; on the other hand, generalized conformal structures and their relation with lightlike manifolds as presented in \cite{BZ17}, which in particular presents the notion of generic and simple lightlike manifold  that will be useful here (see Sections 2 and 6).

\smallskip

The need for a new point of view could be questioned at the light of such number of works on lightlike manifolds. Nevertheless, there is another geometrical situation where the lightlike metrics arise naturally, namely, the study of conformal Riemannian geometry. 
Roughly speaking, starting with a conformal Riemann manifold $(M,[g])$, the space of scales $\mathbf{N}$ consists of the rays of metrics 
$y:=s^{2}g_{x}$ on $T_{x}M$ where $x\in M$ and $s\in \R_{>0}$, \cite{F-G}. There exists a natural projection $\pi:\mathbf{N}\to M$. 
In the terminology of Fefferman and Graham, the \emph{tautological tensor} is the lightlike metric $h$ on $\mathbf{N}$ given by $h(X,Y)= s^{2}g_{x}(T_{y}\pi \cdot X, T_{y}\pi \cdot Y )$, where $X,Y\in T_{y}\mathbf{N}$ and $T\pi$ is the differential map of $\pi$. Then, an ambient metric $g_{_{L}}$ can be defined so that $(M, g_{_{L}})$ is a Lorentzian manifold  that admits $(\mathbf{N}, h)$ as an embedded lightlike hypersurface. From this point of view, lightlike manifolds play an interesting role linking conformal Riemannian geometry and Lorentz geometry. The interest in lightlike manifolds within these theories is purely technical as an intermediate step, whereas in this paper we will reinterpret these constructions focusing on the lightlike manifold. This interplay has inspired Section 6 of this paper. 
The original Fefferman-Graham metric requires certain Ricci-flatness properties and depends only on the conformal class of Riemannian metrics, thus immediately leading to a construction for conformal invariants. An in-depth study and several generalisations of this kind of ideas can be cast in the context of  Cartan geometries \cite{CG03,CS09}.

\smallskip

A Cartan geometry  modelled on a homogeneous space $G/H$ on a manifold $M$  permits to endow $M$  with a differential geometric structure, whose
objects may be thought of as curved analogs of the homogeneous space $G/H$, \cite{CS09} (see Section 3 for a rigorous  definition). The main idea of this paper is to describe the lightlike manifolds as certain Cartan geometries modelled on a suitable homogeneous model. 
Two {\it naive} questions constitute the starting point of this paper: 1)  Is it possible to find an intrinsic approach to lightlike manifolds?
Such {\it intrinsic} approach should rely on reasonable choices with geometric meaning, thus providing local invariants for lightlike manifolds. 2) Does this  {\it intrinsic geometric structure} shed some light on the interplay between lightlike manifolds and another geometric structures? Taking into account that lightlike hypersurfaces are conformally invariant, a certain invariance under conformal changes of the ambient Lorentzain metric will be desirable.

\smallskip

This paper is written from the perspective of Cartan geometries and is organized as follows. 
The first step tackles the search for the homogeneous model $G/H$ for lightlike manifolds, which is dealt with in  Section 2. The meaning of this {\it homogeneous model} for lightlike manifolds deserves some elaboration.
The homogeneous model $\mathcal{N}=G/H$ should be a connected manifold endowed with a $G$-invariant lightlike metric $h$. That is, every left multiplication $\ell_{g}$ by $g\in G$ on $\mathcal{N}$ satisfies $(\ell_{g})^{*}(h)=h$. Moreover, for every $h$-isometry, $f:\mathcal{N}\to \mathcal{N}$, there should be an element $g\in G$ such that $f=\ell_{g}$. Finally, a Liouville type theorem should be satisfied in the sense that
every isometry between two open connected subsets of $\mathcal{N}$ uniquely extends to a global isometry of $\mathcal{N}$.
The study of the groups of isometries of lightlike manifolds in \cite{BFZ} leads our choice of the {\it homogeneous model} for lightlike manifolds. The model for lightlike manifolds will be the 
$(m+1)$-dimensional future lightlike cone $\mathcal{N}^{m+1}$ in Lorentz-Minkowski spacetime $\L^{m+2}$ (Definition \ref{9121}). The future lightlike cone can be described as the homogeneous space $G/H$ of the  M\"{o}bius  group 
$G=PO(m+1,1)$ where $H$ is isomorphic to the group of rigid motions of the $m$-dimensional Euclidean space $\E^{m}$ (see Subsection 2.2). We recall the {\it Liouville type theorem for lightlike geometry} \cite[Theor. 1.1]{BFZ}, which is essential in our approach and justifies our choice of $\mathcal{N}^{m+1}=G/H$ as the model for lightlike geometry (see Corollary \ref{lightKlein}).
Section 2 is also devoted to fix some terminology and notations and introduces the well-known relationship between the future lightlike cone $\mathcal{N}^{m+1}$ and the M\"{obius} sphere $\S^{m}$. 
This section ends with several algebraical properties of the future lightlike cone $\mathcal{N}^{m+1}$ as the {\it Klein geometry} $G/H$, that is, as homogeneous space. We show that $G/H$ is a first order geometry but there is no reductive decomposition for $G/H$ (Lemma \ref{lemalgebra} and Remark \ref{9122}).

\smallskip

Section 3 leans on \cite{CS09} to provide an introduction to Cartan geometries and correspondence spaces. The Maurer-Cartan form of the M\"{o}bius group $G$ is described from a geometrical point of view, by its formulation in terms of the geometry of the future lightlike cone $\mathcal{N}^{m+1}\subset \L^{m+2}$ (Subsection 3.2). The Lie group $G$ is identified with a fiber subbunde of the fiber bundle of linear frames on $\mathcal{N}^{m+1}$ and then, its Maurer-Cartan form is the pull-back to $G$ of the Levi-Civita connection form of $\L^{m+2}$. Thus, the Maurer-Cartan form of $G$ results from the embedding $\mathcal{N}^{m+1}\subset \L^{m+2}$.
From another point of view, $\L^{m+2}$ is essentially the above mentioned ambient space for the conformal sphere $(\S^{m}, [g_{_{\S^{m}}}])$ and $\L^{m+2}$ can be constructed from a section of the projection $\mathcal{N}^{m+1}\to \S^{m}$, (\ref{pro}). 
These two approaches are summarized as follows. In the first one, we get the Maurer-Cartan form of $G$ form the given embedding  $\mathcal{N}^{m+1}\subset \L^{m+2}$. In the second one, we construct an ambient Lorentz manifold starting from the future lightlike cone $\mathcal{N}^{m+1}$ and a section of the projection map. These two points of view will be extended to more general lightlike manifolds in Sections 5 and 6, respectively.

\smallskip

The second step of our programme, developed in Section 4, starts by assuming that there is a Cartan geometry $(p: \mathcal{P}\to N ,\omega)$ on the $(m+1)$-dimensional manifold $N$ of type  $(G,H)$, where $\mathcal{N}^{m+1}=G/H$. Then, we construct a unique lightlike metric $h^{\omega}$ on $N$ as well as a vector field $Z^{\omega}\in  \mathfrak{X}(N)$ that globally spans the radical distribution of $h^{\omega}$ (Theorem \ref{direct}). This result leads to the notion of lightlike Cartan geometry on $N$ (Definition \ref{9123}). This section also includes both the general (and intrinsic) definition of lightlike manifolds and several special cases that appear in \cite{BZ17}. Since the homogeneous model $\mathcal{N}^{m+1}$ is a first order Klein geometry, the total space $\mathcal{P}$ can be identified with all admissible linear frames $Q^{\omega}\subset L(N)$ of the lightlike manifold 
$(N,h^{\omega},Z^{\omega})$, as shown in Equation \eqref{236}.

Section 4 ends with several comments on the Lie group $\mathrm{Aut}(Q^{\omega}, \omega)$ for $N$ a connected manifold (see definition in Section 3).
Every $(F,f)\in \mathrm{Aut}(Q^{\omega}, \omega)$ satisfies $F=Tf$, where $Tf$ is the natural lift of the differential map of $f:N\to N$ to $Q^{\omega}\subset L(N)$. The diffeomorphisms $f$ preserve the lightlike metric $h^{\omega}$ and the vector field $Z^{\omega}$.
In general, $\mathrm{Aut}(Q^{\omega}, \omega)$ is only a proper subgroup of the group of diffeomorphisms preserving $h^{\omega}$ and $Z^{\omega}$, since the latter may even be infinite-dimensional (Corollary \ref{27121} and Example \ref{235}).

\smallskip

The third step  in the study of Cartan geometries is the most subtle, as well as the most important. The question arises as to whether Cartan geometries with a fixed homogeneous model are actually determined by underlying geometric structures on manifolds.
This problem is known as the equivalence problem for the given Cartan geometry. In our setting, this question is stated as follows.
Let $(N,h)$ be a lightlike manifold with radical distribution spanned  by the fixed vector field $Z\in \mathfrak{X}(N)$. Is there a lightlike Cartan geometry $(p: \mathcal{P}\to N ,\omega)$ on $N$ such that the equalities for the lightlike metric $h^{\omega}=h$ and the vector field $Z^{\omega}=Z$ hold? Is there any condition on such lightlike Cartan geometry that permits to assert its unicity? The complete answer  to these questions is not yet ascertained in a general setting, but
Sections 5 and 6 will provide  partial answers to the equivalence problem for lightlike Cartan geometries.

The equivalence problem for lightlike hypersurfaces $(N,h,Z)$ immersed in a given Lorentzian manifold $(M,g)$ is faced in Section 5. First, we recall the notions of expansion function $\lambda$ and null second fundamental form $B_{Z}$, given by Equations (\ref{231}) and (\ref{10121}), respectively.
In this situation, Theorem \ref{main1} and Corollary \ref{umbilical1} are stated as follows.
\begin{quotation}
The pull-back $\omega$ of the Levi-Civita connection from of $(M,g)$ to the $H$-principal fiber bundle of admissible linear frames $Q$
 is a Cartan connection if and only if $(N,h,Z)$ is generic and the expansion function $\lambda$ is nowhere vanishing. In this case, the lightlike metric $h^{\omega}=h(\overline{\nabla}^{g}Z, \overline{\nabla}^{g}Z)$ and the vector field $Z^{\omega}=\frac{1}{\lambda}Z$.

\smallskip

If, in addition, $(N,h, Z)$ is totally umbilical in $(M,g)$ with null second fundamental form $B_{Z}=\rho \, \bar{h}$, then $\omega$ is a Cartan connection if and only if the functions $\rho$ and $\lambda$ are nowhere vanishing. In this case, the lightlike metric $h^{\omega}=\rho^{2}\, h$ and the vector field $Z^{\omega}=\frac{1}{\lambda}Z$. 
\end{quotation}
The totally umbilical lightlike hypersurfaces are invariant under conformal changes of the metric $g$. Thus, up to a conformal change, Corollary \ref{umbilical1} gives a partial answer to the equivalence problem for totally umbilical lightlike manifolds.
Section 5 also includes conditions to assert that a {\it generic and expansive} lightlike hypersurface is locally isometric to the future lightlike cone $\mathcal{N}^{m+1}$ or to the fiber bundle of scales of a conformal Riemannian manifold, Corollary \ref{13121}.

\smallskip

Section 6 provides a construction method of an ambient Lorentzian manifold $(M,g)$ from a given lightlike manifold $(N,h,Z)$. In some sense, the expansion proceeds so that  the original lightlike manifold is a hypersurface of the resulting Lorentzian manifold. This construction is inspired by the ambient metric construction by Fefferman and Graham \cite{F-G} (see Remark \ref{11122}). The procedure for building an ambient Lorentzian manifold can be summarized as follows.
Without loss of generality, the vector field $Z$ that spans the radical distribution can be chosen to be complete so, after a rescaling, we get a right multiplicative action of $\R_{>0}$  on $N^{m+1}$, (\ref{301}). We adopt the technical assumptions that this action is free and proper, which in the case of lightlike hypersurfaces, for instance,  are closely related with the causality properties of the ambient Lorentzian metric (Remark \ref{11121}). In this situation, the orbit space $M=N/\R_{>0}$ is an $m$-dimensional smooth manifold that can be physically interpreted as the light source. Then, from every section $\Gamma: M\to N$, we give a family of Lorentzian metrics on ${\bf M}:=(-\varepsilon, +\varepsilon)\times M \times \R_{>0}$, which are a generalization of the warped product metrics (Remark \ref{warped}). 
The original lightlike manifold $(N,h,Z)$ is now realized as an isometric embedding in ${\bf M}$ with the expansion function---defined in Equation (\ref{231})---chosen as $\lambda=1$. A suitable Ricci flatness condition reduces the family to a one-parametric subset $\{g^{c}\}_{c\in \R}$.
In this setting, the following Theorem \ref{61} is stated.
\begin{quotation}
Let  $(N^{m+1}, h, Z)$ be a generic lightlike manifold with $Z$ complete. Assume the $\R_{>0}$-action deduced by the flow of $Z$ given by Equation (\ref{301}) is free and proper. Then, for every spacelike section $\Gamma:N/\R_{>0}:=M\to N$ and $c\in \R$, the pull-back
$\omega^{c}$ of the Levi-Civita connection form of $g^{c}$ is a lightlike Cartan geometry on $N$.
 The lightlike metric and the vector field deduced from $\omega^{c}$ are $h^{c}=h(\nabla^{c}Z, \nabla^{c}Z)$ and $Z^{c}=Z$, respectively. 
\end{quotation} 
\noindent The paper ends with several examples, showing the application  to simple and generic lightlike manifods \cite{BZ17} and  to geometric flows.

\smallskip

Needless to say, this paper contains more questions than answers. Several ones have been  mentioned above but we would like to end this introductory section with a summary of such open questions:

\begin{itemize}
\item Is it possible to obtain a geometrical characterization of the family of lightlike manifolds $(N,h,Z)$ that can be constructed from Theorem \ref{direct}? We tentatively suggest that a subtle approximation to the notion of lightlike manifolds might be required.

\item Is there any underlying geometric structure for lightlike Cartan geometries that stem from Definition \ref{9123}?

\item  The group $\mathrm{Aut}(\mathcal{P}, \omega)$ for a lightlike Cartan geometry with $N$ connected is a Lie group, which is in general   only a proper subgroup of $\mathrm{Iso}(N,h,Z)$. What is then the geometrical meaning of $\mathrm{Aut}(\mathcal{P}, \omega)$?

\item  In order to recover the original lightlike metric $h$ on a generic lighlike manifold, the following question seems natural. Can the Levi-Civita connection in Theorem  \ref{61} be replaced by a metric connection with torsion on ${\bf M }$ so that $\nabla^{c}Z$ is the identity endomorphism on $TN$? 
\end{itemize}


\section{The guideline example}

\noindent All the manifolds are assumed to be smooth, Hausdorff and satisfying the second axiom of countability. We start by considering the well-known relationship between the 
lightlike cone in the Lorentz-Minkowski spacetime and the M\"{o}bius sphere. A detailed exposition of these facts can be found in \cite{Baum} and \cite{CS09}. In this section we also fix some terminology and notations. We follow closely the treatment given in \cite{CS09}.

Let $\L^{m+2}$ be the $(m+2)$-dimensional Lorentz-Minkowski spacetime, that is, $\L^{m+2}$ is $\R^{m+2}$ endowed
with the Lorentzian metric 
$
\langle\, ,\, \rangle =
-dx_{0}^{2}+\sum_{i=1}^{m}dx_{i}^{2},
$
 where
$(x_{0},x_{1},\dots, x_{m+1})$ are the canonical coordinates of
$\R^{m+2}$.  
Along this paper we assume $m\geq 2$, unless otherwise was stated. 
The future lightlike cone is the hypersurface given by
$$
\mathcal{N}^{m+1}=\left\{v\in \L^{m+2}:\langle v , v\rangle=0, \,\, v_{0}> 0\right\}.
$$
Let us consider $\mathcal{N}^{m+1}/\sim$ the space of lines spanned by elements in $\mathcal{N}^{m+1}$ as a subset of the real projective space $\R P^{m+1}$.
The natural projection
$\pi:\mathcal{N}^{m+1}\to \mathcal{N}^{m+1}/\sim$ is a principal fiber bundle with structure group $\R_{>0}$.  
The vertical distribution $\mathcal{V}$ of $\pi$ is defined at any $v\in \mathcal{N}^{m+1}$ by the kernel of $T_{v}\pi$, the differential map of $\pi$ at $v$. Under the usual identification, we have
$\mathcal{V}(v):=\mathrm{Ker}(T_{v}\pi)=\R \cdot v\subset v^{\perp}= T_{v}\mathcal{N}^{m+1}$. 
An explicit description for the projective future lightlike cone $\mathcal{N}^{m+1}/\sim$ is achieved as follows. Let us consider the $m$-dimensional unit sphere $\S^{m}\subset \E^{m+1}\subset \L^{m+2}$, where $\E^{m+1}$ denotes the $(m+1)$-dimensional Euclidean space embedded in $\L^{m+2}$ at $x_{0}=0$. 
The map $x\mapsto \pi(1,x)$ defines a diffeomorphism between $\S^{m}$ and $\mathcal{N}^{m+1}/\sim$ with inverse $\pi(v_{0}, \cdots , v_{m+1})\mapsto 1/v_{0}(v_{1}, \cdots , v_{m+1})$. Thus, $\S^m$ may be view as the space of lightlike lines in $\L^{m+2}$ and the projection $\pi$ reads as  
\begin{equation}\label{pro}
\mathcal{N}^{m+1} \to \S^{m}, \quad (v_{0},v_{1},\cdots, v_{m+1})\mapsto \frac{1}{v_{0}}(v_{1},\cdots, v_{m+1}).
\end{equation}
In this picture, the projection $\pi$ induces an isomorphism
$
T_{v}\pi:v^{\perp}/(\R \cdot v)\to T_{\pi(v)}\S^{m}
$
which satisfies 
\begin{equation}\label{difP}
T_{v}\pi \cdot w=T_{s\, v}\pi \cdot s\,w,
\end{equation}
for every $s \in \R_{>0}$ and $w\in T_{v}\mathcal{N}^{m+1}$.

At metric level, the future lightlike cone $\mathcal{N}^{m+1}$ inherits from the Lorentzain metric $\langle \,,\,\rangle$ a degenerate symmetric bilinear form $h$ whose radical is exactly the $1$-dimensional vertical distribution $ \mathcal{V}$. The position vector field $\mathcal{Z}\in \mathfrak{X}(\mathcal{N}^{m+1})$ given by $\mathcal{Z}(v)=(v)_{v}$ (where $(\,\,)_{v}$ denotes the usual identification between $T_{v}\R^{m+2}$ and $\R^{m+2}$) spans the vertical distribution $\mathcal{V}$.
Taking into account that $v\in \mathcal{N}^{m+1}$ is a lightlike vector, on every quotient vector space $v^{\perp}/(\R \cdot v)$, the bilinear form $\langle \,,\,\rangle$ induces  a (positive definite) inner product, which can be carried over $T_{\pi(v)}\S^{m}$ by means of $T_{v}\pi$. 
If we replace $v$ by $sv$ for $s\in \R_{>0}$ and using (\ref{difP}), we have that the carried bilinear form on $T_{\pi(v)}\S^{m}$ is multiplied by $s^2$. 
Hence, we get an inner product up to positive multiples on each tangent space $T_{\pi(v)}\S^{m}$.

\smallskip
Recall that two Riemannian metrics $g$ and $g'$ on a manifold $M$ are said to be conformally related when $g'=e^{2\phi}g$ for a smooth function $\phi$ on $M$. The set of all the Riemannian metrics of the form $e^{2\phi}g$ for $\phi \in C^{\infty}(M)$ is called the conformal class $[g]$ of the metric $g$. Let $g_{_{\S^{m}}}$ be the usual round metric of constant sectional curvature $1$ on $\S^m$ and $[g_{_{\S^{m}}}]$ its conformal class. 
The conformal manifold $(\S^{m},[g_{_{\S^{m}}}])$ is called the M\"{o}bius sphere. 
For any (spacelike) section $\sigma:\S^{m}\to \mathcal{N}^{m+1}$ of $\pi$  given by $\sigma(x)=e^{\phi(x)}(1,x)$, the Riemannian metric $\sigma^{*}\langle \,\,,\,\,\rangle =e^{2\phi}g_{_{\S^{m}}}\in [g_{_{\S^{m}}}]$. Conversely,
for every metric $g\in [g_{_{\S^{m}}}]$, there exists a section $\sigma$ of $\pi$ with $g=\sigma^{*}\langle \,\,,\,\,\rangle$.

A Riemannian metric on a manifold $M$ is a section of the fiber bundle of positive definite symmetric bilinear forms $\mathrm{Sym}^{+}(TM)$ on $M$. Following ideas in \cite{BZ17}, the future lightlike cone $\mathcal{N}^{m+1}$ can be seem from a different point of view. 
Namely, every $v\in \mathcal{N}^{m+1}$ with $\pi(v)=x$ gives an element $c(v)\in \mathrm{Sym}^{+}(T_{x}\S^{m})$ requiring that $T_{v}\pi: v^{\perp}/(\R\cdot v)\to T_{x}\S^{m}$ is an isometry.
Thus, $\mathcal{N}^{m+1}$ is realized as a $(m+1)$-dimensional submanifold of $\mathrm{Sym}^{+}(T\S^{m})$ and the following diagram commutes
$$
\begin{array}{rcl}
 \mathcal{N}^{m+1} & \stackrel{c\ \ }\longrightarrow &\mathrm{Sym}^{+}(T\S^{m}) \\  \searrow &  &  \swarrow  \\   & \S^{m} &
\end{array}
$$
For every $x\in \S^{m}$, $c(\pi^{-1}(x))=\{s^{2}\,g_{_{\S^{m}}}\vert _{x}: s>0\}$ is the positive ray spanned by $g_{_{\S^{m}}}$ at $x\in \S^{m}$.
The triple $(\mathcal{N}^{m+1},h, \mathcal{Z})$ is our guideline example for the study of lightlike manifolds. 

\subsection{ $\mathcal{N}^{m+1}$ as homogeneous model for lightlike manifolds}

We turn now to the Klein geometry picture of the these facts. That is, we give a description of the above construction in terms of homogeneous spaces.
Let us consider $\L^{m+2}$ with a fixed basis $\mathcal{B}=(\ell,e_{1},e_{2},\cdots , e_{m}, \eta)$ with $\ell\in \mathcal{N}^{m+1}$ and such that the corresponding matrix to $\langle \,,\,\rangle$ with respect to $\mathcal{B}$ is given by $$S:=\left(\begin{array}{ccc} 0 & 0 & 1\\ 0 & \mathrm{I}_{m} & 0 \\ 1 & 0 & 0 \end{array} \right), $$
where $\mathrm{I}_{m}$ is the identity matrix with $m$-rows. 
Let $O(m+1, 1)=\{\sigma\in \mathrm{Gl}(m+2, \R):\sigma^{t}S\sigma=S\}$ be the pseudo-orthogonal group of signature $(m+1,1)$, that is, the Lorentz group. 
In our setting, it is better to replace the Lorentz group by the M\"{o}bius  group 
$PO(m+1,1):=O(m+1, 1)/\{\pm \mathrm{Id}\}$. In fact, we prefer to consider $\S^{m}$ as the space of lightlike lines $\mathcal{N}^{m+1}/\sim$ so that the Lie group $PO(m+1, 1)$ acts effectively on $\S^{m}$.

From now on, let $G:=PO(m+1,1):=O(m+1, 1)/\{\pm \mathrm{Id}\} $ be the M\"{o}bius  group. For $\sigma\in O(m+1,1)$, let us denote by $[\sigma]$ the corresponding element in $G$. 
The natural action of $O(m+1,1)$ on $\L^{m+2}$  leaves the lightlike vectors invariant and induces a transitive smooth action on the whole lightlike cone
$$
C^{m+1}=\{v\in \L^{m+2}:\langle v, v\rangle=0, v\neq 0\}=(-\mathcal{N}^{m+1})\cup \mathcal{N}^{m+1}.
$$
Let $C^{m+1}/\Z_{2}$ be the quotient identifying antipodal points. Every class  $[v]=\{\pm v\}\in C^{m+1}/\Z_{2}$ has an unique representative $v^{+}\in \mathcal{N}^{m+1}$. This property permits to identify $\mathcal{N}^{m+1}$ with $C^{m+1}/\Z_{2}$. Under such identification, the Lie group $G$ acts on the future lightlike cone as follows
 \begin{equation}\label{action-cone}
 G\times \mathcal{N}^{m+1}\to \mathcal{N}^{m+1},\quad ([\sigma],v)\mapsto \lambda^{\mathcal{N}^{m+1}}_{[\sigma]}(v):=(\sigma v)^{+}.
 \end{equation}
 This action is transitive on $\mathcal{N}^{m+1}$ and thus we get a diffeomorphism $\mathcal{N}^{m+1}\cong G/H,
$
where $H$ is the isotropy subgroup of the element $\ell \in \mathcal{N}^{m+1}$. 
\begin{definition}\label{9121}
The manifold $\mathcal{N}^{m+1}\cong G/H$ is called the $(m+1)$-dimensional lightlike homogeneous model manifold.
 \end{definition}

The action of the group $G$ descends on $\S^{m}\cong \pi(\mathcal{N}^{m+1})$ by $\lambda^{\S^{m}}_{[\sigma]}(\pi(v))=\pi(\sigma v)$, where $\pi$ is given in (\ref{pro}). This action is effective and transitive and we get a diffeomorphism $\S^{m}\cong G/P$, where $P$ is the isotropy subgroup of $\pi(\ell)$.  The group $P$ is called the Poincar\'{e} conformal group.
The group $G$ acts on $\S^{m}$ by conformal transformations (see details in \cite{Baum}) and the explicit formula is as follows
$$
\lambda^{\S^{m}}_{[\sigma]}(x)=\frac{1}{\sigma(1,x)^{+}_{0}}\left(\sigma(1,x)^{+}_{1},\cdots ,\sigma(1,x)^{+}_{m+1} \right),
$$
where $\sigma(1,x)^{+}=(\sigma(1,x)^{+}_{0},\cdots ,\sigma(1,x)^{+}_{m+1})\in \mathcal{N}^{m+1}.$ 
For $m\geq 2$, the Lie group of global conformal transformations of the M\"{o}bius  sphere $(\S^{m}, [g_{_{\S^{m}}}])$ is the M\"{o}bius  group $G$. Moreover, for $m\geq 3$, every conformal transformation betweeen two connected open subsets of $(\S^{m},  [g_{_{\S^{m}}}])$ is the restriction of a {\it left multiplication} $\lambda^{\S^{m}}_{[\sigma]}$ by an element $[\sigma]\in G$. This local property does not hold for $m=2$. For the case $m=1$, every diffeomorphim of the circle $\S^{1}$ gives a conformal transformation. 
As a consequence of these rigidity properties, 
the following {\it Liouville type theorem for lightlike geometry} has been obtained  in \cite[Theor. 1.1]{BFZ}.

 \begin{theorem}\label{isometrias}
For $m\geq 2$, the group of (global) isometries $\mathrm{Iso}(\mathcal{N}^{m+1}, h)$ is the M\"{o}bius group $G$ acting as in (\ref{action-cone}).
\begin{enumerate}
\item [i)] For $m\geq 3$, every isometry betweeen two connected open subsets of $\mathcal{N}^{m+1}$ is the restriction of the action by an element of $G$.
\item [ii)] The group of local isometries of $\mathcal{N}^{3}$ is the group of local conformal transformations of $\S^{2}$.
\item [iii)] Every diffeomorphism of the circle $\S^{1}$ corresponds to an isometry of $\mathcal{N}^{2}\subset \L^{3}$.
\end{enumerate}
\end{theorem}

\begin{remark}\label{Klein}
{\rm The method of the proof for \cite[Theor. 1.1]{BFZ} relies on the following fact. The lightlike cone $(\mathcal{N}^{m+1},h)$ is isometric to $\S^{m}\times \R$ endowed with $e^{2t}g_{_{\S^{m}}}\oplus 0$. This observation permits to give an explicit group isomorphism between $\mathrm{Iso}(\mathcal{N}^{m+1}, h)$ and the group $\mathrm{Conf}(\S^{m})$ of conformal transformations of $\S^m$ as follows.

For technical reasons, we prefer here to identify $\mathcal{N}^{m+1}$ with $\S^{m}\times \R_{>0}$ by meaning of
\begin{equation}\label{i}
i:\S^{m}\times \R_{>0} \to \mathcal{N}^{m+1}, \quad (x,s)\mapsto (s, sx),
\end{equation}
where we have used canonical coordinates of $\R^{m+2}$.
The bilinear form  $h$ and the vector field $\mathcal{Z}$ correspond with $\widehat{h}=s^{2}g_{\S^{m}}\oplus 0$ and $\widehat{\mathcal{Z}}=s\partial_{s}$, respectively. 
After the suitable change of variable, it is shown in \cite[Theor. 1.1]{BFZ} that
$
f\in \mathrm{Iso}(\S^{m}\times \R_{>0}, \widehat{h})
$ if and only if 
$f(x, s)= (\Phi(x), se^{-\phi(x)})$ where $\Phi\in \mathrm{Conf}(\S^{m})$ and $\Phi^{*}(g_{_{\S^{m}}})=e^{2\phi}g_{_{\S^{m}}}$.
For every $\Phi\in \mathrm{Conf}(\S^{m})$, let $f^{\Phi}\in \mathrm{Iso}(\S^{m}\times \R_{>0}, \widehat{h})$ be given by $f^{\Phi}(x, s)= (\Phi(x), se^{-\phi(x)})$, as above. 
The isomorphism is achieved as follows
$$
\mathrm{Conf}(\S^{m})\to \mathrm{Iso}(\S^{m}\times \R_{>0}, \widehat{h}),\quad  \Phi\mapsto f^{\Phi}.
$$

We know that $\Phi\in \mathrm{Conf}(\S^{m})$ if and only if $\Phi=\lambda^{\S^{m}}_{[\sigma]}$ for $[\sigma]\in G$ ($m\geq 2$) and the conformal factor is determined by $e^{2\phi(x)}=1/(\sigma(1,x)^{+}_{0})^{2}$.

The following diagram commutes
\begin{equation}
\begin{CD}
\S^{m}\times \R_{>0}@ > f^{\Phi} >> \S^{m}\times \R_{>0} \\
@V  i VV @VV i V \\
\mathcal{N} ^{m+1} @>>\lambda^{\mathcal{N}^{m+1}}_{[\sigma]}  > \mathcal{N}^{m+1}
\end{CD}
\end{equation}
 In fact, we compute for $(x,s)\in \S^{m}\times \R_{>0}$,
 $$
 i\left(f^{\Phi}(x,s)\right)=i\left((\Phi(x), se^{-\phi(x)})\right)=\left(se^{-\phi(x)} ,se^{-\phi(x)}\lambda^{\S^{m}}_{[\sigma]}(x) \ \right)=
 s (\sigma(1,x))^{+}=\lambda^{\mathcal{N}^{m+1}}_{[\sigma]}\left(i(x,s)\right).
 $$
We also have $T_{(x,s)}f^{\Phi}\cdot \widehat{\mathcal{Z}}(x,s)= \widehat{\mathcal{Z}} (f^{\Phi}(x,s))$. Therefore, every element of 
$\mathrm{Iso}(\mathcal{N}^{m+1}, h)$ preserves the vector field $\mathcal{Z}$.

}
 \end{remark}
As a consequence of Remark \ref{Klein} and \cite[Theor. 1.1]{BFZ}, we obtain the following description of the Klein geometry determined by $G$ on $\mathcal{N}^{m+1}$.

\begin{corollary}\label{lightKlein}
As Klein geometry $\mathcal{N}^{m+1}\cong G/H$ and for $m\geq 2$, the left mutiplications by elements of the Lie group $G$ agree with the diffeomorphisms of $\mathcal{N}^{m+1}$ preserving $h$ and $\mathcal{Z}$. For $m\geq 3$, every diffeomorphism betweeen two connected open subsets of $\mathcal{N}^{m+1}$ which preserves $h$ and $\mathcal{Z}$ is the restriction of the left multiplication by an element of $G$.
\end{corollary}

\subsection{Algebraic description of the model}

The elements of the Poincar\'{e} conformal group $P$ correspond with the elements of $G$ of the form
$$
P=\left\{\left[\begin{array}{ccc} \lambda & -\lambda w^{t}g & -\frac{\lambda}{2}\| w^{t}\|^2 \\ 0 & g& w \\ 0 & 0 & \lambda^{-1} \end{array} \right]: \lambda \in \R \setminus \{0\},\,\, w\in \R^{m},\,\, g\in O(m)\right\},
$$
where $\|w^{t}\|^2=g_{_{\R^{m}}}(w,w)$ denotes the square of the usual Euclidean norm on $\R^{m}$ \cite[Prop. 1.6.3]{CS09}. The Lie group $H$ corresponds to the elements of $P$ with $\lambda=\pm 1$. The group of rigid motions of the $m$-dimensional Euclidean space $\E^m$ is naturally isomorphic to $H$. That is, $H$ can be identified with the semi-direct product $ \R^{m}\rtimes O(m)\cong \mathrm{Euc}(\E^{m})$ by means of
\begin{equation}\label{iso1}
\left(\begin{array}{cc} 1 &  0  \\ w & g  \end{array} \right)
\mapsto \left[\begin{array}{ccc} 1 & -  w^{t}g & -\frac{1}{2}\| w^{t}\|^2\\ 0 & g & w \\ 0 & 0 & 1 \end{array} \right]\in H.
\end{equation}
The Lie group $H$ is a closed normal subgroup of $P$ and $P/H\cong \R_{>0}$.
The natural projection $$G/H \to G/P, \quad gH\mapsto gP$$ gives a principal fiber bundle with structure group $P/H$ and corresponds to the projection given in (\ref{pro}).

At Lie algebras level, we have for the Lie algebra $\mathfrak{g}$ of $PO(m+1,1)$, 
$$
\mathfrak{g}= \left\{\left(\begin{array}{ccc} a & Z & 0\\ X & A & -Z^{t} \\ 0 & -X^{t} & -a \end{array} \right): a\in \R, X\in \R^{m}, Z\in (\R^{m})^{*}, A\in \mathfrak{o}(m)\right\}.
$$
The decomposition
$
\mathfrak{g}=\mathfrak{g}_{-1}\oplus \mathfrak{g}_{0} \oplus \mathfrak{g}_{1}
$ where
$$
\mathfrak{g}_{-1}=\left\{\left(\begin{array}{ccc} 0 & 0 & 0\\ X & 0 & 0 \\ 0 & -X^{t} & 0 \end{array} \right)\right\}, 
\mathfrak{g}_{0}=\left\{\left(\begin{array}{ccc} a & 0 & 0\\ 0 & A & 0 \\ 0 & 0 & -a \end{array} \right)\right\}, 
\mathfrak{g}_{1}=\left\{\left(\begin{array}{ccc} 0 & Z & 0\\ 0 & 0 & -Z^{t} \\ 0 & 0 & 0 \end{array} \right)\right\}$$
defines a grading $\mathfrak{g}$, that is, $[\mathfrak{g}_{i},\mathfrak{g}_{j}]\subset \mathfrak{g}_{i+j}$, where $\mathfrak{g}_{k}=0$ when $k\notin \{-1,0,1\}$.
The Lie algebra of $P$ is given by $\mathfrak{p}=\mathfrak{g}_{0}\oplus \mathfrak{g}_{1}$ and the Lie algebra $\mathfrak{h}$ is the Lie subalgebra of $\mathfrak{p}$ determinated by $a=0$ \cite[Sec. 1.6.3]{CS09}. We have $\mathfrak{g}_{0}=\mathfrak{z}(\mathfrak{g}_{0})\oplus [\mathfrak{g}_{0}, \mathfrak{g}_{0}]$, where $\mathfrak{z}(\mathfrak{g}_{0})$ is the center of the Lie algebra $\mathfrak{g}_{0}$ and $\mathfrak{h}=[\mathfrak{g}_{0}, \mathfrak{g}_{0}]\oplus \mathfrak{g}_{1}$.
The quotient vector space $\mathfrak{g}/ \mathfrak{h}$ can be identified with $\R \times \R^{m}$ as follows 
\begin{equation}\label{151}
\left(\begin{array}{ccc} a & Z & 0\\ X & A & -Z^{t} \\ 0 & -X^{t} & -a \end{array} \right)+ \mathfrak{h} \mapsto \left(
\begin{array}{c}
a \\ X
\end{array}
\right)\in \R \times \R^{m}.
\end{equation}
Let us recall that the quotient adjoint representation $\underline{\mathrm{Ad}}:H\longrightarrow \mathrm{Gl}(\mathfrak{g}/\mathfrak{h)}$  is defined as 
\begin{equation}\label{141}
 \underline{\mathrm{Ad}}[\sigma]:\mathfrak{g}/\mathfrak{h} \to \mathfrak{g}/\mathfrak{h},\quad Y+\mathfrak{h}\mapsto \mathrm{Ad}[\sigma](Y)+\mathfrak{h},\end{equation}
where $\mathrm{Ad}$ denotes the adjoint representation of $G$.

\smallskip

In a general setting, let $\mathfrak{g}\subset \mathfrak{gl}(m,\R)$ be a Lie subalgebra.
For $k=0,1,2,\cdots $, let $\mathfrak{g}_{k}$ be the space of symmetric $(k+1)$-multilinear mappings $$t:\R^{m}\times \cdots \times \R^{m}\to \R^{m}$$ such that, for each $v_{1}, \cdots , v_{k}\in \R^{m}$, the linear transformation $v\in \R^{m}\to t(v, v_{1}, \cdots , v_{k})\in \R^{m}$ belongs to $\mathfrak{g}$. The space $\mathfrak{g}_{k}$ is called the $k$-th prolongation of $\mathfrak{g}$. If there is a $k$ such that $\mathfrak{g}_{k}=0$, the Lie algebra $\mathfrak{g}$ is said to be of finite order $k$, otherwise $\mathfrak{g}$ is of infinite type \cite[p. 4]{Kob}.
Recall that $\mathfrak{g}\subset \mathfrak{gl}(m,\R)$ is of infinite type if it contains a matrix of range $1$ as an element \cite[Proposition 1.4]{Kob}.

\smallskip

Here are several properties of the homogeneous manifold $G/H$ as a Klein geometry.
\begin{lemma}\label{lemalgebra}\footnote{I am grateful to Professor Cristina Draper for the indications for the proof of this Lemma.}
Let us consider $\mathcal{N}^{m+1}=G/H$ the $(m+1)$-dimensional lightlike homogeneous model manifold with $m\geq 2$. Then,
\begin{itemize}
\item[i)] The Lie algebra $\mathfrak{h}$ does not admit any reductive complement in $\mathfrak{g}$.
\item[ii)] Under the identification (\ref{151}), the quotient adjoint representation satisfies
$$
\underline{\mathrm{Ad}}[\sigma]\left(\begin{array}{c}
a \\ X
\end{array}
\right)=\left(\begin{array}{c}
a-  g_{_{\R^{m}}}( w, gX) \\  gX
\end{array}
\right),
$$
where 
 $[\sigma]=\tiny{\left[\begin{array}{ccc} 1 & - w^{t}g & -\frac{1}{2}\| w^{t}\|^2\\ 0 & g & w \\ 0 & 0 & 1 \end{array} \right]}\in H$.
\item[iii)] $\underline{\mathrm{Ad}}$ is injective. That is, the Klein geometry $G/H$ is a first order Klein geometry.

\item[iv)] The Lie algebra $\mathfrak{h}$ is of infinite type.
\end{itemize}
\end{lemma}
\begin{proof}
Let us assume there is a reductive decomposition $\mathfrak{g}=\mathfrak{h}\oplus \mathfrak{m}$, that is, we have $\mathrm{Ad}(H)(\mathfrak{m})\subset \mathfrak{m}$ and in particular $[\mathfrak{h}, \mathfrak{m}]\subset \mathfrak{m}$.  Consider $\mathfrak{h}_{0}=[\mathfrak{g}_{0}, \mathfrak{g}_{0}]\cong \mathfrak{o}(m)$, which is a simple Lie algebra. The simplicity of $\mathfrak{h}_{0}$ says that every $\mathfrak{h}_{0}$-module is completely reducible. In particular, the three 
 $\mathfrak{h}_{0}$-modules,  $\mathfrak{g}$, $\mathfrak{h}$ and $\mathfrak{m}$, decompose as a direct sum of irreducible $\mathfrak{h}_{0}$-modules. In such decomposition for $\mathfrak{g}$,
$$\mathfrak{g}=\mathfrak{g}_{-1}\oplus \mathfrak{z}(\mathfrak{g}_{0}) \oplus  \mathfrak{h}_{0} \oplus \mathfrak{g}_{1},$$
the following $\mathfrak{h}_{0}$-modules appear:
\begin{itemize}
\item the adjoint module, $\mathfrak{h}_{0}$;
\item two natural $\mathfrak{h}_{0}$-modules, $\mathfrak{g}_{-1}$ and its dual, $\mathfrak{g}_{1}$, which are isomorphic;
\item
a trivial one-dimensional $\mathfrak{h}_{0}$-module, the center $\mathfrak{z}(\mathfrak{g}_{0})$, spanned by 
$E=\tiny\left(\begin{array}{ccc} 1 & 0 & 0\\ 0 & 0 & 0 \\ 0 & 0 & -1 \end{array} \right).$ \newline ($E$ is the grading element). 
\end{itemize}
Also,   $\mathfrak{h}=\mathfrak{h}_{0} \oplus \mathfrak{g}_{1}$ provides the decomposition of $\mathfrak{h}$ as a direct sum of irreducible $\mathfrak{h}_{0}$-modules, one of them of adjoint type and the other one isomorphic to the natural $\mathfrak{o}(m)$-module $\R^{m}$. Hence, the complement $\mathfrak{m}=\mathfrak{m}_1\oplus \mathfrak{m}_2$ should be the sum of two irreducible $\mathfrak{h}_{0}$-modules,  a trivial module $\mathfrak{m}_1$ and   a natural one $\mathfrak{m}_2$. This means that $\mathfrak{m}_2\subset \mathfrak{g}_{1} \oplus \mathfrak{g}_{-1}$  and that $\mathfrak{m}_1=\mathfrak{z}(\mathfrak{g}_{0})$, since it is the unique trivial module in the decomposition of $\mathfrak{g}$. Hence $E\in \mathfrak{m}$. Now, $[E,\mathfrak{g}_1]\subset [\mathfrak{g}_0,\mathfrak{g}_1]\subset \mathfrak{g}_1\subset \mathfrak{h}$  but also $[E,\mathfrak{g}_1]\subset [\mathfrak{m},\mathfrak{h}]\subset \mathfrak{m}$. Thus $[E,\mathfrak{g}_1]\subset\mathfrak{h}\cap\mathfrak{m}=0$, which gives a contradiction since 
$(\textrm{ad}\,E)\vert_{\mathfrak{g}_{j}}=j\,\textrm{id}\vert_{\mathfrak{g}_{j}}$ for $j\in \{-1,0,1\}$.
The contradiction comes from our assumption on the existence of  $\mathfrak{m}$.

An easy computations shows item $ii)$ and then, item  $iii)$ follows immediately.  For item $iv)$, let us consider the Lie algebras isomorphism induced from (\ref{iso1}),
$$
\mathfrak{euc}(\E^{m})\to  \mathfrak{h},\quad \scriptsize{\left(\begin{array}{cc} 0 & 0 \\ Z & A   \end{array} \right)\mapsto \left(\begin{array}{ccc} 0 & Z & 0\\ 0 & A & -Z^{t} \\ 0 & 0 & 0 \end{array} \right)}.
$$
There are matrices of range $1$ in $\mathfrak{euc}(\E^{m})$ and \cite[Proposition 1.4]{Kob} gives the result.
\end{proof}

\begin{remark}\label{9122}
{\rm Every Cartan geometry with reductive homogeneous model admits a covariant derivative \cite[Chap. 3]{Sharpe}. Moreover, every first-order and reductive Cartan geometry on a manifold $M$ corresponds to a principal (Ehresmann) connection on a principal fiber bundle of linear frames on $M$ \cite[App. A]{Sharpe}. Consequently, the point of Lemma \ref{lemalgebra} is in the assertion: The Lie algebra $\mathfrak{h}$ does not admit any reductive complement in $\mathfrak{g}$ (see Remarks \ref{21121} and \ref{15121}).
}
\end{remark}

\begin{remark}\label{170320A}
{\rm For each $[\sigma] \in H$,  we can compute the action of $\mathrm{Ad}[\sigma]$ on the non-reductive complement $\mathfrak{g}_{-1}\otimes \mathfrak{z}(\mathfrak{g}_{0})$ as follows.
For each $1\leq i\leq m$, let us consider
\begin{equation}\label{170320B}
E_{i}=\scriptsize{\left(\begin{array}{ccc} 0 & 0 & 0\\  \bar{e}_{i}  & 0 & 0 \\ 0 &  -\bar{e}_{i}^{t}  & 0 \end{array} \right)}\in \mathfrak{g}_{-1},
\end{equation}
where $(\bar{e}_{1}, \cdots , \bar{e}_{m})$ is the canonical basis of $\R^{m}$. For
 $[\sigma]\in H$ as in Lemma \ref{lemalgebra} $ii)$, straightforward computations give
$$
\mathrm{Ad}[\sigma](E_{i})=
$$
$$
\scriptsize{\left(\begin{array}{ccc} -g_{_{\R^{m}}}(w, g \bar{e}_{i}) & 0 & 0\\   g \bar{e}_{i} & 0 &  0 \\ 0 &  -\bar{e}_{i}^{t}g^{-1}  &  g_{_{\R^{m}}}(w, g \bar{e}_{i}) \end{array} \right)}+ 
\scriptsize{\left(\begin{array}{ccc} 0 & -g_{_{\R^{m}}}(w, g \bar{e}_{i})w^{t}+ \frac{1}{2}\| w^{t} \|^{2}(g \bar{e}_{i})^{t} & 0\\  0  & g(\bar{e}_{i}w^{t}- w \bar{e}_{i}^{t})g^{-1} & g_{_{\R^{m}}}(w, g \bar{e}_{i})w- \frac{1}{2}\| w^{t} \|^{2}g \bar{e}_{i} \\ 0 &  0  & 0 \end{array} \right)}
$$
and for the grading element $E$,
$$
\mathrm{Ad}[\sigma](E)=E+ \tiny{\left(\begin{array}{ccc} 0 & w^{t} & 0\\  0  & 0 & -w \\ 0 &  0  & 0 \end{array} \right)}.
$$

}

\end{remark}

\section{Cartan connections and correspondence spaces}

In this section, we concisely give an introduction to Cartan connections and correspondence spaces construction. 
The general reference here is \cite{CS09}. Unless otherwise was stated, in this section $G$ and $H$ are arbitrary Lie groups.

For a principal fiber bundle $p:\mathcal{P}\to M$ with structure group $H$, let us denote $r^{\sigma}$ for the right translation on $\mathcal{P}$ by $\sigma \in H$ and $\xi_{X}\in \mathfrak{X}(\mathcal{P})$ for the fundamental vector field coresponding to $X\in \mathfrak{h}=\mathrm{Lie}(H)$. That is,
$$
\xi_{X}(u):=\frac{d}{dt}\vert_{0}(u \cdot \mathrm{exp}(tX)).
$$

Let $G$ be a Lie group and $H\subset G$ a closed subgroup and let $\mathfrak{g}$ be Lie algebra of $G$. A Cartan geometry of type $(G,H)$ on a manifold $M$ consists of
\begin{enumerate}
\item A principal fiber bundle $p:\mathcal{P}\to M$ with structure group $H$.

\item A $\mathfrak{g}$-valued $1$-form $\omega \in \Omega^{1}(\mathcal{P}, \mathfrak{g})$, called the Cartan connection, such that for every $u\in \mathcal{P}, \sigma\in H$ and $X\in \mathfrak{h}$, 
\begin{enumerate}
\item $
\omega(u):T_{u}\mathcal{P}\to \mathfrak{g}
$
is a linear isomorphism,
\item  $(r^{\sigma})^{*}(\omega)=\mathrm{Ad}(\sigma^{-1})\circ \omega$ and
\item $\omega(u)(\xi_{X}(u))=X.$
 \end{enumerate} 
\end{enumerate}
For every $X\in \mathfrak{g}$, the constant vector field $\omega^{-1}(X)\in \mathfrak{X}(\mathcal{P})$ is defined by the condition $\omega (\omega^{-1}(X)(u))=X$ for all $u\in \mathcal{P}$. Thus we obtain a linear map $\omega^{-1}:\mathfrak{g}\to \mathfrak{X}(\mathcal{P})$. In general $\omega^{-1}$ is not a Lie algebras homomorphism, the obstruction to this property is codified in the curvature form $K:=d\omega+ \frac{1}{2}[\omega, \omega]\in \Omega^{2}(\mathcal{P}, \mathfrak{g})$. In an equivalent way, we have
$$
K(\xi ,\eta)=d\omega (\xi, \eta)+[\omega(\xi), \omega (\eta)], \quad \xi, \eta\in \mathfrak{X}(\mathcal{P}). 
$$
A Cartan connection $\omega$ is said to be torsion free when $K$
takes values in $\mathfrak{h}$.
All the information of $K$ is contained in the curvature function $\kappa : \mathcal{P}\to \Lambda^{2}\mathfrak{g}^{*}\otimes \mathfrak{g}$ defined as $\kappa(u)(X,Y)=K(\omega^{-1}(X)(u),\omega^{-1}(Y)(u)).$ Taking into account that $K$ is horizontal, that is, $K$ vanishes when we insert a vertical tangent vector, the curvature function may be view as $\kappa : \mathcal{P}\to \Lambda^{2}(\mathfrak{g}/\mathfrak{h})^{*}\otimes \mathfrak{g}$ \cite[Lem. 1.5.1]{CS09}.

The first example of Cartan geometry of type $(G,H)$ is the canonical projection $p:G\to G/H$ endowed with the (left) Maurer-Cartan form $\omega_{G}\in \Omega^{1}(G, \mathfrak{g})$. This example is called the homogeneous model for the Cartan geometries of type $(G,H)$. The Maurer-Cartan equation implies that the homogeneous model of any Cartan geometry has zero curvature \cite[Sec. 1.2.4]{CS09}. In this case, for each $X\in \mathfrak{g}$, the constant vector field $\omega^{-1}_{G}(X)$  is the left invariant vector field $L_{X}$ on $G$ with $L_{X}(e)=X$.

A Cartan connection provides a description the tangent fiber bundle of the base manifold $M$ as follows.
Let us consider $
\underline{\mathrm{Ad}}:H\to \mathrm{Gl}(\mathfrak{g}/\mathfrak{h})
$ the quotient adjoint representation $(\ref{141})$.
 For each $u\in \mathcal{P}$ with $p(u)=x\in M$, there is a canonical linear isomorphism 
$\phi_{u}:T_{x}M \to \mathfrak{g}/\mathfrak{h}$ such that the following diagram conmutes
\begin{equation}\label{isomor}
\begin{CD}
T_{u}\mathcal{P}  @>\omega(u)>> \mathfrak{g}\\
@V T_{u} p  VV @VV \mathrm{proj.} V \\
T_{x}M @>>\phi_{u} \cong > \mathfrak{g}/\mathfrak{h}
\end{CD}\quad \quad \text{with } \phi_{u\sigma}=\underline{\mathrm{Ad}}(\sigma^{-1})\phi_{u} \text{ for all }\sigma\in H.
\end{equation}
There is a canonical isomorphism of vector fiber bundles,
\begin{equation}\label{142}
TM \cong \mathcal{P}\times_{_{H}} \mathfrak{g}/\mathfrak{h},\quad (x,v)\in TM\mapsto [u, \phi_{u}(v)],
\end{equation}
where $p(u)=x$ \cite[Theor. 3.15]{Sharpe}.

An isomorphism of Cartan geometries $(p:\mathcal{P}\to M, \omega)$ and $(p':\mathcal{P}'\to M', \omega')$ with the same model $(G,H)$ is a principal fiber bundle isomorphism $(F,f)$ 
$$
 \begin{CD}
\mathcal{P}  @>F >> \mathcal{P}'\\
@V p VV @VV p' V \\
M @>>  f > M'
\end{CD}
$$
such that $F^{*}(\omega ')=\omega$.
That is, $F$ is a diffeomorphism from $\mathcal{P}$ to $\mathcal{P}'$ (and so is $f$) such that $F\circ r^{\sigma}=r^{\sigma}\circ F$ for all $\sigma\in H$ and $F^{*}(\omega ')=\omega$. 
In the particular case that both Cartan geometries are the same, an isomorphism is called an automorphism.
The group $\mathrm{Aut}(\mathcal{P}, \omega)$ of all automorphisms of the Cartan geometry $(p:\mathcal{P}\to M, \omega)$ over a connected manifold $M$ is a Lie group (may be with uncountably connected
components) and dimension at most $\mathrm{dim}(G)$ \cite[Theor. 1.5.11]{CS09}.

A Cartan geometry $p:\mathcal{P}\to M$ of type $(G,H)$ has curvature form $K=0$ if and only if every point $x\in M$ has an open neighborhood $U\subset M$ such that $(p:p^{-1}(U)\to U, \omega\vert_{U})$ is isomorphic to the restriction of the homogeneous model $(G\to G/H, \omega_{G})$ to an open neighborhood of $o:=eH$  \cite[Prop. 1.5.2]{CS09}.

\subsection{Correspondence spaces}
 Let us recall the construction called correspondence spaces for general Cartan geometries.  
 This subsection derives from \cite[Sec. 1.5.13]{CS09}. Let $(p:\mathcal{P}\to M, \omega)$ be a Cartan geometry of type $(G,P)$ and $H\subset P$ a closed subgroup. The correspondence space for $H\subset P$ is defined to be $\mathcal{C}(M):=\mathcal{P}/H$. 
The natural projection $\pi:\mathcal{C}(M)\to M$ is a fiber bundle over $M$ with fiber the homogeneous space $P/H$
and $(\pi:\mathcal{P}\to \mathcal{C}(M), \omega)$ is a Cartan geometry of type $(G,H)$ on $\mathcal{C}(M)$ \cite[Prop. 1.5.13]{CS09}. We have the following commutative diagram
\begin{equation}\label{corres}
\begin{array}{rcl}
   & \mathcal{P}& \\  \swarrow   &  & \searrow \\ \mathcal{C}(M) & \stackrel{\pi }{\longrightarrow} & \,\,M
\end{array}
\end{equation}
We get from (\ref{142}) a description of the tangent fiber bundle of $\mathcal{C}(M)$ as  $\mathcal{P}\times _{_{H}}(\mathfrak{g}/ \mathfrak{h})$. In these terms, the vertical distribution of $\pi$ corresponds to $\mathcal{P}\times _{_{H}}(\mathfrak{p}/ \mathfrak{h})\subset \mathcal{P}\times _{_{H}}(\mathfrak{g}/ \mathfrak{h})=T\mathcal{C}(M)$ \cite[Prop. 1.5.13]{CS09}.
The corresponding curvature function
$
k^{\mathcal{C}(M)}:\mathcal{P}\to \Lambda^{2}(\mathfrak{g}/\mathfrak{h})^{*}\otimes \mathfrak{g}
$
satisfies $k^{\mathcal{C}(M)}(u)(X+\mathfrak{h}, *)=0$ for every $X\in \mathfrak{p}$ and $u\in \mathcal{P}$.

The following characterization of correspondence spaces \cite[Th. 1.5.14]{CS09} will be used in Corollary \ref{13121}.
Let $G$ be a Lie group and $H\subset P\subset G$ closed subgroups.
Let $(\pi:\mathcal{P}\to N, \omega)$ be a Cartan geometry of type $(G,H)$ such that the distribution $\mathcal{V}(N):= \mathcal{P}\times _{H}(\mathfrak{p}/ \mathfrak{h}) \subset TN$ is integrable. 
A (local) twistor space for $N$ is a smooth manifold $M$ together with an open subset $U\subset N$ and a surjective submersion $f:U\to M$ such that $\mathcal{V}_{x}N=\mathrm{Ker}(T_{x}f)$. That is, a (local) space of leaf for the foliation defined by $\mathcal{V}(N),$ \cite[Def. 1.5.14]{CS09}. 
Let us assume the curvature function 
$
k^{N}:\mathcal{P}\to \Lambda^{2}(\mathfrak{g}/\mathfrak{h})^{*}\otimes \mathfrak{g}
$
satisfies $k^{N}(u)(X+\mathfrak{h}, *)=0$ for every $X\in \mathfrak{p}$ and $u\in\mathcal{P}$. Then, for any sufficiently small local twistor space $f:U\to M$, there is a Cartan geometry of type $(G,P)$ on $M$ such that the restricted Cartan geometry $(p:p^{-1}(U) \to U, \omega\vert_{U})$ of type $(G,H)$ is isomorphic to an open subset of the correspondence space $\mathcal{C}(M)$. In the case $P/H$ is connected this Cartan geometry on $M$ is unique.

For example, if we take $G$  the M\"{o}bius, $P$ the Poincar\'{e} conformal group and $H\subset P$ as in Section 2 and consider the Cartan geometries $G\to G/H$ and $G\to G/P$ both endowed with the Maurer-Cartan form of the Lie group $G$. Then, the correspondence space for $H\subset P$ is $\mathcal{C}(\S^{m})=\mathcal{N}^{m+1}$. The diagram (\ref{corres}) reads as follows
$$
\begin{array}{rcl}
   & G & \\  \swarrow   &  & \searrow \\ \mathcal{N}^{m+1}& \stackrel{\pi }{\longrightarrow} & \,\, \,\,\S^{m}
\end{array}
$$

\subsection{The Maurer-Cartan form of $G$}

From now on, the groups $G, P$ and $H$ are the M\"obius group, the Poincar\'{e} conformal group and $H$ corresponds to the elements of $P$ with $\lambda=\pm 1$ as in section 2.

We end this section taking a look on the Maurer-Cartan form $\omega\in \Omega^{1}(G, \mathfrak{g})$ of the M\"obius group $G$.
We will focus on the geometric description of $\omega$ in terms of the lightlike homogeneous model manifold $\mathcal{N}^{m+1}$. This point of view will lead us to the construction for more general lightlike manifolds (Def. (\ref{52})) in Sections 5 and 6. 

\smallskip

Let $p:G\to G/H =\mathcal{N}^{m+1}$ be the projection. For every $[\sigma]=[v\vert e_{1}\vert\cdots \vert e_{m}\vert y ]\in G$ ($[\sigma]$ is written as column vectors), we have $p[\sigma]=v^{+}\in \mathcal{N}^{m+1}$.
If we write $\tau: O(m+1,1)\to G$ for the projection $\sigma \mapsto [\sigma]$, then we get
$$
T_{[\sigma]}G= \Big\{ T_{\sigma}\tau \cdot \xi:\xi\in \mathcal{M}_{m+2}(\R), \,\, \xi^{t}S\sigma +\sigma^{t}S\xi=0 \Big\}
$$
and
$T_{[\sigma]}p\cdot T_{\sigma}\tau\cdot \xi = \xi_{0}\in T_{v^{+}}\mathcal{N}^{m+1}$, where $\xi=(\xi_{0}\vert \xi_{1}\vert  \cdots \vert \xi_{m}\vert \xi_{m+1})$.
The Maurer-Cartan $\omega$ of the Lie group $G$ is given by
\begin{equation}\label{MCG}
\omega[\sigma]: T_{[\sigma]}G \to \mathfrak{g}, \quad  T_{\sigma}\tau \cdot \xi \mapsto \sigma^{-1}\xi=S \sigma^{t}S\xi.
\end{equation}
For each $0\leq j \leq m+1$, the $j$th-column of $\omega[\sigma](T_{\sigma}\tau \cdot \xi )$ are the coordinates of $\xi_{j}$ with respect to the basis $(v^{+}, e_{1}, \cdots , e_{m}, y)$ where $(v^{+}\vert e_{1}\vert \cdots \vert e_{m}\vert y )\in [\sigma]$.

Following \cite[Exercise 3.21]{Sharpe} and as a consequence of item $iii)$ in Lemma \ref{lemalgebra},
the Lie group $G$ is identified with an $H$-principal fiber subbundle of the principal fiber bundle of the linear frames $L(\mathcal{N}^{m+1})$ as follows. Let us fix the basis 
\begin{equation}\label{150120}
\left((1,0)^{t}, (0,\bar{e}_{1})^{t}, \cdots ,(0,\bar{e}_{n})^{t}\right)
\end{equation}
for $\mathfrak{g}/\mathfrak{h}\cong  \R \times \R^{m}$ where $(\bar{e}_{1},\cdots, \bar{e}_{m})$ is the canonical basis of $\R^{m}$. The admisible frames at $v^{+}\in \mathcal{N}^{m+1}$ are 
$$
\left(\phi_{[\sigma]}^{-1}(\bar{e}_{0}),\cdots, \phi_{[\sigma]}^{-1}(\bar{e}_{m})\right),\,\,\textrm{ for all }[\sigma]\in G \textrm{ with }p[\sigma]=v^{+}.$$
The linear isomorphisms $
\phi_{[\sigma]}:T_{v^{+}}\mathcal{N}^{m+1} \to \mathfrak{g}/\mathfrak{h}
$ are given in (\ref{isomor}).
Now, it is easy to check from (\ref{MCG}) that 
$$
u:G\to L(\mathcal{N}^{m+1}), \quad [\sigma]\mapsto u[\sigma]=(v^{+}, e_{1},\cdots ,e_{m} )\subset L_{v^{+}} (\mathcal{N}^{m+1}),
$$
where $(v^{+}\vert e_{1}\vert \cdots \vert e_{m}\vert y )\in [\sigma]$ is a fiber bundle morphism.

\smallskip

Let $\gamma\in \Omega^{1}(\L^{m+2}\times O(m+1,1),\mathfrak{g})$ be the (Ehresmann) principal connection corresponding to the Levi-Civita connection $\nabla^{0}$ of $\L^{m+2}$. Thus, we have 
$$
\gamma(v,\sigma):T_{(v,\sigma)}\left(\L^{m+2}\times O(m+1,1)\right), \quad (x,\xi)\mapsto \omega'(\xi),
$$
where $\omega'$ is the Maurer-Cartan form of $O(m+1,1)$. 
The inclusion $\psi: \mathcal{N}^{m+1} \hookrightarrow \L^{m+2}$ is lifted to 
$$\Psi:G\to \L^{m+2}\times O(m+1,1),\quad
\Psi(v^{+}, e_{1},\cdots , e_{m})=\left( v^{+}, (v^{+}, e_{1},\cdots , e_{m}, y) \right).$$ Here $y\in \L^{m+2}$ is determined by the condition $ (v^{+}, e_{1},\cdots , e_{m}, y) \in O(m+1,1)$ and the Lie group $G$ is identified with a subset of $L(\mathcal{N}^{m+1})$ by means of the fiber bundle morphism $u$.
Then, the Maurer-Cartan form $\omega$ of $G$ satisfies
$
\omega=\Psi^{*}(\gamma).
$

\smallskip
Let us denote by $-\overline{\nabla}^{0}\mathcal{Z}$ is the Weingarten endomorphism corresponding to the normal vector field $\mathcal{Z}$
along the inclusion $\mathcal{N}^{m+1}\subset \L^{m+2}$.
We know that $\overline{\nabla}^{0}\mathcal{Z}$ is the identity map on every $T_{v}\mathcal{N}^{m+1}$.
Thus, taking into account that $
\omega=\Psi^{*}(\gamma),
$
we can write
$$
\big[\omega[\sigma](T_{\sigma}\tau \cdot \xi)\big]_{\mathfrak{z}(\mathfrak{g}_{0})}=\langle \overline{\nabla}^{0}_{\xi_{0}}\mathcal{Z}, y \rangle \, \, \textrm{ and } \big[\omega[\sigma](T_{\sigma}\tau \cdot \xi)\big]_{-1}=\left(\langle \overline{\nabla}^{0}_{\xi_{0}}\mathcal{Z}, e_{1}\rangle , \cdots ,  \langle \overline{\nabla}^{0}_{\xi_{0}}\mathcal{Z}, e_{m}\rangle \right)^{t},
$$
where the subscripts denote the corresponding projection from $\mathfrak{g}$ on $\mathfrak{z}(\mathfrak{g}_{0})$ and $\mathfrak{g}_{-1}$ (compare with Lemma \ref{232}).
On the other hand, the resulting equality $\mathcal{L}_{\mathcal{Z}}h=2h$, where $\mathcal{L}$ is the Lie derivative, implies that the future lightlike cone is generic with expansion function $\lambda=1$, in the terminology of Section 5. These properties will lead Section 5 for lightlike hypersurfaces.

\smallskip

In section 6, our approach will be intrinsic. That is, a Lorentzian manifold, which contains the lightlike manifold as hypersurface, cannot be a priori given. Nevertheless, under certain assumptions on the lightlike manifold, an ambient Lorentzian manifold will be constructed. Let us explain how it works for  the future lightlike cone $\mathcal{N}^{m+1}$. Thus, we show how the Lorentz-Minkowski spacetime $\L^{m+2}$ can be constructed from a section of the trivial $\R_{>0}-$principal fiber bundle given in (\ref{pro}). This is a very special case of the Fefferman-Graham ambient metric construction \cite{F-G}.
In fact, the Lorentz-Minkowski spacetime $\L^{m+2}$  essentially is the Fefferman-Graham ambient space for the conformal sphere $(\S^{m}, [g_{_{\S^{m}}}])$ as follows.
Fix the section $\Gamma:\S^{m}\to \mathcal{N}^{m+1}$ given by $x\mapsto (1,x)$.
This section produces the global trivialization $i$ given in (\ref{i}) and $\Gamma^{*}(h)=g_{_{\S^{m}}}$. Now, for sufficiently small $\varepsilon>0$, the product manifold ${\bf M}=(-\varepsilon, +\varepsilon) \times \S^{m} \times \R_{>0}\cong (-\varepsilon, +\varepsilon) \times \mathcal{N}^{m+1}$ is endowed with the Lorentzian metric  \cite[Formula 7.12]{F-G}
\begin{equation}\label{29122}
g_{(\rho, x, s)}=ds\otimes d(\rho s)+d(\rho s)\otimes ds +s^{2}(1+\rho/2)^2g_{_{\S^{m}}}.
\end{equation}
The Lorentz metric $g$ satisfies $g(\partial_{s}, \partial_{s})=2\rho$, $g(\partial_{\rho}, \partial_{\rho})=0$ and $g(\partial_{s} \partial_{\rho})=s$ and $g$ is Ricci flat. The future lightlike cone $\mathcal{N}^{m+1}$ is isometric to the hypersurface in ${\bf M}$ at $\rho=0$ and the metric $g$ is also realized as the induced metric from the immersion
$$
\alpha:{\bf M} \to \L^{m+2}, \quad (\rho ,x,s)\mapsto \Big((1-\rho/2) s,
(1+\rho/2)sx)\Big).
$$
This point of view will be adopted in Section 6.

\section{Lightlike Cartan geometries}
\begin{definition}\label{52}
A lightlike manifold is a pair $(N^{m+1},h)$ where $N$ is a $(m+1)$-dimensional smooth manifold with $m\geq 2$ and furnished with a lightlike metric $h$. That is, $h$ is a symmetric $(0,2)$ tensor field on $N$ such that
\begin{enumerate}
\item $h(u,u)\geq 0$ for all $u\in T_{y}N$. 
\item $\mathrm{Rad}(h_{y})=\{u\in T_{y}N: h(u,-)=0\}$ defines a $1$-dimensional distribution on $N$.
\end{enumerate}
\end{definition}

\noindent The lightlike metric $h$ induces a bundle-like Riemannian metric $\overline{h}$ on the vector fiber bundle $\mathcal{E}:=TN/\mathrm{Rad}(h)$ as follows $\overline{h}([u],[v])=h(u,v)$ for all $u,v \in T_{y}N$, here the brackets denote the classes in the quotient space $T_{y}N/\mathrm{Rad}(h_{y})$.
The radical distribution $\mathrm{Rad}(h)$ is said to be orientable when there exists a vector field $Z\in \mathfrak{X}(N)$ which globally spans the radical distribution $\mathrm{Rad}(h)$ and we write $(N,h,Z)$ to fix a such vector field $Z$.
In this case, the Lie derivative $\mathcal{L}_{Z}h$ provides a well-defined bilinear map on each $T_{y}N/\mathrm{Rad}(h_{y})$, denotes by
$\mathcal{L}_{Z} \overline{h}$ and given by $\mathcal{L}_{Z} \overline{h}([u],[v])=\mathcal{L}_{Z}h(u,v).$ The bundle-like metric $\overline{h}$ determines a morphism $A_{Z}$ on the vector fiber bundle $\mathcal{E}$ by the condition 
\begin{equation}\label{211}
\mathcal{L}_{Z} \overline{h}([u],[v])=2\overline{h}(A_{Z}[u],[v])
\end{equation} 
for all $u,v \in T_{y}N$. Following the terminology in \cite{BZ17}, the lightlike manifold $(N,h,Z)$ is said to be generic when $A_{Z}$ is an isomorphism on $\mathcal{E}$. The genericity condition is independent of the choice of $Z$ orienting $\mathrm{Rad}(h)$.

Every isomorphism $D$ of the vector fiber bundle $\mathcal{E}$ on the identity of $N$ produces the new lightlike metric $h_{_{D}}(u, v):= \overline{h}(D[u], D[v])$. The lightlike metrics $h$ and $h_{_{D}}$ shares the radical distribution. Under the genericity condition on $(N,h,Z)$, the lightlike metric
\begin{equation}\label{210120}
h_{A_{Z}^{-1}}(u, v)= \overline{h}(A_{Z}^{-1}[u], A_{Z}^{-1}[v])
\end{equation}
 will have a relevant role in Theorem \ref{main1} .
\smallskip

Let us denote by $\mathrm{Iso}(N,h,Z)=\{f\in \mathrm{Diff}(N): f^{*}(h)=h\,\, \mathrm{ and } \,\,Tf\cdot Z=Z\circ f\}$. Taking into account that $T_{y}f\cdot Z(y)=Z(f(y))$ for all $y\in N$, the differential map $Tf$ induces a vector fiber bundle morphism $\overline{Tf}:\mathcal{E}\to \mathcal{E}$ on the map $f:N \to N$. We have $
f\circ \mathrm{Fl}^{Z}_{\varepsilon}=\mathrm{Fl}^{Z}_{\varepsilon} \circ f
$
for small enough $\varepsilon >0$, where $\mathrm{Fl}^Z$ denotes the flow of the vector field $Z$. An easy computation shows that
$$
f^{*}(\mathcal{L}_{Z}\overline{h})([v], [w])=\frac{d}{d\varepsilon }\mid_{\varepsilon =0}\Big(h(T_{f(y)}\mathrm{Fl}^{Z}_{\varepsilon}\cdot T_{y} f\cdot v,T_{f(y)}\mathrm{Fl}^{Z}_{\varepsilon}\cdot T_{y} f\cdot w)\Big)=
$$
$$\frac{d}{d\varepsilon }\mid_{\varepsilon =0}\Big(h( T_{y}\mathrm{Fl}^{Z}_{\varepsilon}\cdot v,T_{y}\mathrm{Fl}^{Z}_{\varepsilon}\cdot w)\Big)=
2\overline{h}(A_{Z}[v],[w]).
$$
That is, we get $\overline{Tf}\circ A_{Z}=A_{Z}\circ \overline{Tf}$ and therefore, for generic  lightlike manifolds $(N,h,Z)$, we have $\mathrm{Iso}(N,h,Z)= \mathrm{Iso}(N,h_{A_{Z}^{-1}},Z)$.

\begin{definition}\label{82}
A local twistor space for $(N,h,Z)$ is a smooth manifold $M$ together with an open subset $U\subset N$ and a surjective submersion $\pi:U\to M$ such that $\mathrm{Ker}(T_{y}\pi)=\mathrm{Rad}(h_{y})$ for all $y\in U$ (compare with \cite[Def. 1.5.14]{CS09}).
\end{definition} 

\begin{remark}\label{291}
{\rm Since $\mathrm{Rad}(h)$ is a $1$-dimensional distribution such local twistor spaces $\pi :U\to M$ always exist. We can pick
$(U,\phi=(s, r_{1},\cdots , r_{m}))$ a cubic coordinate system with $\phi(U)=(-\varepsilon, +\varepsilon)^{m+1}$ such that $Z\mid_{U}=\frac{\partial}{\partial s}$ and $h=\sum_{i,j=1}^{m}h_{ij}\,dr_{i}\, \otimes dr_{j}$. 
In this case,  the generic condition reads as $\mathrm{det}(\frac{\partial h_{i\, j}}{\partial s}\mid_{1\leq i,j\leq m})\neq 0$.
For $M=(-\varepsilon, +\varepsilon)^{m}$, we have $(-\varepsilon, +\varepsilon) \times M= U$ and every section $\Gamma : M \to U$ endows $M$ with the Riemannian metric $\Gamma^{*}(h)$. 
In the terminology of \cite{BZ17}, every local twistor space can be view as a generalized conformal structure on the manifold $M$. 
The generic condition holds when $\mathcal{L}_{Z}h=2 \tau h$ for $\tau$
a nowhere vanishing smooth function on $N$. In this case, we have $A_{Z}=\tau \,\mathrm{Id}_{\mathcal{E}}$ and $h$ has the following local expression
$$
h=\exp\left(2\int \tau \,ds \right)\sum_{i,j=1}^{m}\,C_{ij}(r_{1},\cdots , r_{m})\,dr_{i}\, \otimes dr_{j}.
$$
}
\end{remark}

\smallskip

Let $(N,h,Z)$ be a lightlike manifold. The set $Q\subset L(N)$ of all admissible linear frames of $N$ is defined as
\begin{equation}\label{laQ}
Q=\big\{b=(Z(y),e_{1}, \cdots , e_{m})\in L_{y}(N): y\in N, \,\, h(e_{i}, e_{j})=\delta_{ij}\big\}.
\end{equation}
Thus, $Q$ is a reduction of the structure group $Gl(m+1, \R)$ of $L(N)$ to the Lie group $H$.
Explicitly, we have
\begin{equation}\label{D201}
b\cdot [\sigma]=\left(Z(y), \sum_{i=1}^{m}g_{i1}(-w_{i}Z(y)+e_{i}),\cdots , \sum_{i=1}^{m}g_{im}(-w_{i}Z(y)+e_{i})\right),
\end{equation}
where $[\sigma]=\tiny{\left[\begin{array}{ccc} 1 & -  w^{t}g & -\frac{1}{2}\| w^{t}\|^2\\ 0 & g & w \\ 0 & 0 & 1 \end{array} \right]\in H}.
$

\subsection{Cartan geometries with model $(G,H)$}

Recall that
$G=PO(m+1,1)$ denotes the M\"{o}bius  group and 
$$
H=\left\{\left[\begin{array}{ccc} 1 & -w^{t}g & -\frac{1}{2}\| w^{t}\|^2 \\ 0 & g& w \\ 0 & 0 & 1 \end{array} \right]: \,\, w\in \R^{m},\,\, g\in O(m)\right\}.
$$
The quotient vector space $\mathfrak{g}/ \mathfrak{h}$ is identified with $\R \times \R^{m}$ as in (\ref{151}) and furnished with the lightlike metric
$$q\Big(\left(
\begin{array}{c}
a \\ X
\end{array}
\right), \left(
\begin{array}{c}
b \\ Y
\end{array}
\right)\Big)=g_{_{\R^{m}}}( X , Y).$$
The constant vector field $(1,0)^{t}\in \mathfrak{X}(\mathfrak{g}/ \mathfrak{h})$ spans the radical distribution $\mathrm{Rad}(q)$. Moreover, it follows for all $[\sigma]\in H$ (Lemma \ref{lemalgebra} $ii)$) 
$$
q\Big(\underline{\mathrm{Ad}}[\sigma]\left(
\begin{array}{c}
a \\ X
\end{array}
\right),\underline{\mathrm{Ad}}[\sigma]\left(
\begin{array}{c}
b \\ Y
\end{array}
\right) \Big)=q\Big(\left(
\begin{array}{c}
a \\ X
\end{array}
\right), \left(
\begin{array}{c}
b \\ Y
\end{array}
\right)\Big) \,\, \text{and}\,\,
\underline{\mathrm{Ad}}[\sigma]\left(
\begin{array}{c}
1 \\ 0
\end{array}
\right)=\left(
\begin{array}{c}
1 \\ 0
\end{array}
\right).
$$
 \begin{theorem}\label{direct}
Every Cartan geometry $(p:\mathcal{P}\to N^{m+1}, \omega)$ of type $(G,H)$ determines a lightlike metric $h^{\omega}$ on the base manifold $N$ and a vector field $Z^{\omega}\in \mathfrak{X}(N)$ that globally generates the distribution $\mathrm{Rad}(h^{\omega})$.  \end{theorem}
 \begin{proof}
Let us fix $b\in \mathcal{P}$ with $p(b)=y\in N$. We use the linear isomorphism
$\phi_{b}:T_{y}N \to \mathfrak{g}/\mathfrak{h}$ given in (\ref{isomor}) to transport the bilinear form $q$ to the bilinear form $h^{\omega}_{b}$ on $T_{y}N$ defined by $h^{\omega}_{b}(v,w)=q(\phi_{b}(v), \phi_{b}(w))$. For all $[\sigma]\in H$, we get 
$$
q(\phi_{b}(v), \phi_{b}(w))=q(\underline{\mathrm{Ad}}[\sigma^{-1}]\phi_{b}(v) , \underline{\mathrm{Ad}}[\sigma^{-1}]\phi_{b}(w) )=q(\phi_{b\sigma}(v), \phi_{b\sigma}(w)).
$$
Thus, $h^{\omega}_{b}$ does not depend on the choice of $b\in p^{-1}(y)$ and we write $h^{\omega}$. 
From the commutative diagram (\ref{isomor}), we have that 
\begin{equation}\label{233}
q(\omega (b)(\xi)+\mathfrak{h}, \omega (b)(\eta)+\mathfrak{h})=h^{\omega} (T_{b}p\cdot \xi , T_{b}p\cdot \eta).
\end{equation}
Taking into account that $p$ is a submersion, we get that $h^{\omega}$ is a smooth tensor and therefore a lightlike metric on $N$.
The linear isomorphism 
$\phi_{b}$ also transports the vector field $(1,0)^{t}\in \mathfrak{X}(\mathfrak{g}/ \mathfrak{h})$ to $\mathfrak{X}(N)$ as follows
$Z^{\omega}_{b}(y)=\phi_{b}^{-1}(1,0)^{t}\in T_{y}N.$
In fact, we have for all $[\sigma]\in H$,
$$
Z^{\omega}_{b}(y)=\phi_{b}^{-1} \left(
\begin{array}{c}
1 \\ 0
\end{array}
\right)=\phi_{b}^{-1}\underline{\mathrm{Ad}}[\sigma] \left(
\begin{array}{c}
1 \\ 0
\end{array}
\right)=(\phi_{b[\sigma]})^{-1} \left(
\begin{array}{c}
1 \\ 0
\end{array}
\right)=Z^{\omega}_{b[\sigma]}(y)
$$
and $Z^{\omega}_{b}(y)$ does not depend on the choice of $b\in p^{-1}(y)$ and we write $Z^{\omega}$. Since $p$ is a submersion, to check that $Z^{\omega}\in \mathfrak{X}(N)$, we only need to show that $Z^{\omega}\circ p:\mathcal{P}\to TN$ is smooth. According to the definition, we have
\begin{equation}\label{234}
Z^{\omega}(p(b))=T_{b}p\cdot \omega (b)^{-1}(E),
\end{equation}
where $E=\tiny{\left(\begin{array}{ccc} 1 &  0 & 0\\ 0 & 0 & 0 \\ 0 & 0 & -1 \end{array} \right)}\in \mathfrak{g}$. That is $Z^{\omega}\circ p= Tp\circ \omega^{-1}(E)$, for $\omega^{-1}(E)\in \mathfrak{X}(\mathcal{P})$ the constant vector field corresponding to $E$. Therefore,  $Z^{\omega}$ is a smooth vector field on $N$ and clearly spans the radical distribution of $h^{\omega}$.
\end{proof}

\begin{definition}\label{9123}
A Cartan geometry $(p:\mathcal{P}\to N, \omega)$ modeled on $(G,H)$ is said to be a lightlike Cartan geometry on $N$.
\end{definition}
\noindent Given an isomorphism $(F,f)$ of lightlike Cartan geometries $(p:\mathcal{P}\to N, \omega)$ and $(p':\mathcal{P}'\to N', \omega')$, it is easily deduced that $f^{*}(h^{\omega '})=h^{\omega}$ and $Tf\cdot Z^{\omega}=Z^{\omega '}\circ f$. The converse is not true,  in general  (see Corollary \ref{27121} and Example \ref{235}).

\begin{corollary}\label{140120}
Let $(\pi:\mathcal{P}\to M, \omega)$ be a Cartan geometry with model $(G,P)$ and consider the correspondence space 
$\mathcal{C}(M)=\mathcal{P}/H$ for $H\subset P$.
Then, $(p:\mathcal{P}\to \mathcal{C}(M), \omega)$  is a lightlike Cartan geometry on $\mathcal{C}(M)$.
\end{corollary}

\begin{remark}\label{292}
{\rm A  Cartan geometry $(\pi:\mathcal{P}\to M, \omega)$ with model $(G,P)$ determines a Riemannian conformal structure $[g]$ on the base manifold $M$ (see for instance \cite[Lemma 7.2.3]{Sharpe}). Once we fix a representative $g\in [g]$, we obtain a global trivialization of $\mathcal{C}(M)\cong M\times \R_{>0}$. In terms of this identification, the lightlike metric in Corollary \ref{140120} reads as $h^{\omega}=s^{2}g\oplus 0$ and the radical distribution is spanned by $Z^{\omega}=s \partial_{s}$, \cite{F-G}. Fefferman and Graham call this lightlike metric $h^{\omega}$ as the tautological symmetric $2$-tensor \cite{F-G}. The correspondence space $\mathcal{C}(M)$ can also be identified with the fiber bundle of scales of the conformal manifold $(M, [g])$ \cite[Section 1.6.5]{CS09}. As well as the case of the future lightlike cone, we have $\mathcal{L}_{s \partial_{s}}h^{\omega}=2 h^{\omega}$ and so $(\mathcal{C}(M), h^{\omega}, s\partial_{s})$ is generic.
}
\end{remark}
\begin{remark}
{\rm On other matters, Liouville type theorem for lightlike geometry \cite{BFZ} seems to indicate that lightlike Cartan geometries on $3$-dimensional manifolds play a special role. Let us recall that a M\"{o}bius structure on a conformal Riemannian surface $(M, [g])$ is a map
$$
[g]\to \mathcal{T}_{(0,2}(M),\quad  g\mapsto P^{g}
$$
such that $P^{g}$ is symmetric, $\mathrm{trace}_{g}P^{g}= K^{g}$, where $K^{g}$ denotes the Gauss curvature of $g$, and the following conformal transformation law is satisfied
$
P^{e^{2\sigma}g}=P^{g}-\frac{1}{2}\| \nabla^{g} \sigma \|_{g}^{2}g- \mathrm{Hess}^{g}(\sigma)+ d\sigma\otimes d\sigma,
$ \cite{Ran}.
The M\"{o}bius structure plays the role of the Schouten tensor for dimension $m=2$. Although may be interesting to relate $3$-dimensional lightlike manifolds with M\"{o}bius structures on conformal Riemannian surfaces, this is out of the scope of this paper. For example, a M\"{o}bius structure can be obtained from certain Weingarten endomorphism for spacelike surfaces through the future lightlike cone $\mathcal{N}^{3}$ (see details in \cite{PR13}).}
\end{remark}

Let
$(p:\mathcal{P}\to N, \omega)$ be a lightlike Cartan geometry on $N$ and $Q^{\omega}$ the set of admissible linear frames given in (\ref{laQ}) for the lightlike manifold $(N,h^{\omega}, Z^{\omega})$ in Theorem \ref{direct}. 
Following \cite[Exercise 3.21]{Sharpe} and Lemma \ref{lemalgebra} $iii)$, the total space $\mathcal{P}$ can be identified with the set of all admissible linear frames $Q^{\omega}$ by means of the following fiber bundle isomorphism 
\begin{equation}\label{236}
\mathcal{P}\to Q^{\omega},\quad b\mapsto \Big(Z^{\omega}(p(b)), \phi_{b}^{-1} \left(
\begin{array}{c}
0 \\ \bar{e}_{1}
\end{array}
\right), \cdots, \phi_{b}^{-1} \left(
\begin{array}{c}
0 \\ \bar{e}_{m}
\end{array}
\right)\Big)
\end{equation}
where  $\left((1,0)^{t}, (0,\bar{e}_{1})^{t}, \cdots ,(0,\bar{e}_{n})^{t}\right)$ is the fixed basis for 
$\mathfrak{g}/\mathfrak{h}\cong  \R \times \R^{m}$ in (\ref{150120}). Hence $Q^{\omega}$ inherits an action of $H$ on the right by means of (\ref{236}) as follows \cite[Exercise 3.21]{Sharpe}, 
$$
\Big(Z^{\omega}(p(b)), \phi_{b}^{-1} \left(
\begin{array}{c}
0 \\ \bar{e}_{1}
\end{array}
\right), \cdots, \phi_{b}^{-1} \left(
\begin{array}{c}
0 \\ \bar{e}_{m}
\end{array}
\right)\Big)\cdot [\sigma]=\Big(Z^{\omega}(p(b)), \phi_{b[\sigma]}^{-1} \left(
\begin{array}{c}
0 \\ \bar{e}_{1}
\end{array}
\right), \cdots, \phi_{b[\sigma]}^{-1} \left(
\begin{array}{c}
0 \\ \bar{e}_{m}
\end{array}
\right)\Big)
$$
for all $[\sigma] \in H$.
A direct computation from Lemma \ref{lemalgebra} shows that 
$$
\phi_{b[\sigma]}^{-1} \left(
\begin{array}{c}
0 \\ \bar{e}_{j}
\end{array}
\right)=\phi_{b}^{-1}\circ \underline{\mathrm{Ad}}[\sigma]\left(
\begin{array}{c}
0 \\ \bar{e}_{j}
\end{array}
\right)=\phi_{b}^{-1}\left(
\begin{array}{c}
-g_{_{\R^{m}}}(w, g\bar{e}_{j}) \\ g\bar{e}_{j}
\end{array}
\right)=\left( \sum_{i=1}^{m}g_{ij}\big(-w_{i}Z(p(b))+e_{i}\big)\right),$$
where $[\sigma]=\tiny{\left[\begin{array}{ccc} 1 & -  w^{t}g & -\frac{1}{2}\| w^{t}\|^2\\ 0 & g & w \\ 0 & 0 & 1 \end{array} \right]}\in H$ and $\phi_{b}^{-1} \left(
\begin{array}{c}
0 \\ \bar{e}_{i}
\end{array}
\right)=e_{i}$ for all $1\leq i\leq m$.
Therefore, (\ref{236}) defines a $H$-principal fiber bundle isomorphism with the action on $Q^{\omega}$ given in (\ref{D201}).
In the sequel, we omit this identification and write $(p:Q^{\omega}\to N, \omega)$ for a lightlike Cartan geometry on the manifold $N$. The Cartan connection is now view as  $\omega\in \Omega^{1}(Q^{\omega}, \mathfrak{g})$.

\bigskip

We end this Section looking for a description for the Lie group $\mathrm{Aut}(Q^{\omega}, \omega)$ when the base manifold $N$ is connected. To start with, let us recall the soldering form $\theta\in \Omega^{1}(L(N), \R^{m+1})$ 
given by $\theta (u)(\xi)=u^{-1}(T_{u}p \cdot \xi)$ for all $\xi \in T_{u}L(N)$. The soldering form $\theta$
is $Gl(m+1,\R)$-equivariant. That is, for all $u\in L(N)$, $ g\in Gl(m+1, \R)$ and $\xi\in T_{u}L(N)$ we have,
\begin{equation}\label{181}
\theta(u\cdot g)(T_{u}r^{g}\cdot \xi)=g^{-1}\circ \theta(u)(\xi),
\end{equation}
where $r^{g}$ is the action on the right by $g$ on $L(N)$. The soldering form $\theta$ is induced on $Q^{\omega}$ and 
\begin{equation}\label{182}
(\theta\mid_{\mathcal{Q^{\omega}}})(b)(\xi):=b^{-1}(T_{b}p \cdot \xi)=\mathrm{proj}\circ\omega(b)(\xi),
\end{equation}
where $\xi \in T_{b}Q^{\omega}$, $\mathrm{proj}: \mathfrak{g}\to \mathfrak{g}/\mathfrak{h}$ is the natural projection and $\mathfrak{g}/\mathfrak{h}\cong \R \times \R^{m}$ as in (\ref{151}). Therefore, the Cartan connection $\omega$ can be written as 
$$
\omega=(\theta\mid_{Q^{\omega}})\oplus \omega_{[\mathfrak{g}_{0}, \mathfrak{g}_{0}]}\oplus \omega_{1},
$$
where the subscripts denote the projections from $\mathfrak{g}$ on $[\mathfrak{g}_{0}, \mathfrak{g}_{0}]$ and $\mathfrak{g}_{1}$, respectively.

For every smooth map $f: N \to N$, we denote by $Tf$ the lift $Tf: L(N)\to L(N)$ deduced from the tangent map of $f$. That is $$Tf\cdot u:=(T_{y}f\cdot v_{1}, T_{y}f\cdot v_{2}, \cdots ,T_{y}f\cdot v_{m+1} ),$$ for $u=( v_{1}, \cdots , v_{m+1})$ a linear frame at $y\in N$.
The following argument is a slight variation of \cite[Remark 1.5.3]{CS09}, we include it here for the sake of completeness. 
Let $(F,f)$ be an automorphism of the lightlike Cartan geometry $(p:Q^{\omega}\to N, \omega)$.
The condition $F^{*}(\omega)=\omega$ and (\ref{182}) imply
 $F^{*}(\theta\mid_{Q^{\omega}})=\theta\mid_{Q^{\omega}}$.  Thus, for every $\xi \in T_{b}Q^{\omega}$ with $p(b)=y$, one has
 $$
b^{-1}(T_{b}p\cdot \xi)=\theta(b)(\xi)=\theta(F(b))(T_{b}F\cdot \xi)= F(b)^{-1}(T_{F(b)}p\cdot T_{b}F\cdot \xi)=F(b)^{-1}(T_{y}f\cdot T_{b}p\cdot \xi).
$$ 
Taking into account that $p$ is a submersion, we get $b^{-1}(v)=F(b)^{-1}(T_{y}f\cdot v)$ for all $v \in T_{y}N$
which implies $F=Tf$.

Let  $(p:Q^{\omega}\to N, \omega)$ be a lightlike Cartan geometry on a connected manifold $N$.
For every $(Tf, f)\in \mathrm{Aut}(Q^{\omega}, \omega)$, the condition $Tf:Q^{\omega}\to Q^{\omega}$ directly gives that $f\in \mathrm{Iso}(N,h^{\omega},Z^{\omega})$. In this way, we get the group homomorphism
\begin{equation}\label{18121}
\beta:\mathrm{Aut}(Q^{\omega}, \omega)\to \mathrm{Iso}(N,h^{\omega},Z^{\omega}), \quad (Tf,f)\mapsto f.
\end{equation}
\noindent The map $\beta$ is not an isomorphism, in general (see Example \ref{235}). 

\begin{remark}\label{21121}
{\rm The distribution $\mathcal{H}=\omega^{-1}\left(\mathfrak{g}_{-1}\oplus \mathfrak{z}(\mathfrak{g}_{0})\right)\subset TQ^{\omega}$ is complementary to the vertical distribution $\mathcal{V}$ of $p:Q^{\omega}\to N$. Nevertheless, $\mathcal{H}$ does not give a (Ehresmann) principal connection. In fact, for every $b\in Q^{\omega}$, $\xi \in \mathcal{H}(b)$ and $[\sigma]\in H$, we have
$$
\omega(b[\sigma])\left(T_{b}r^{[\sigma]}\cdot \xi\right)=\mathrm{Ad}[\sigma]^{-1}\left[\omega(b)(\xi)\right]\subset \mathrm{Ad}[\sigma]^{-1}\left(\mathfrak{g}_{-1}\oplus \mathfrak{z}(\mathfrak{g}_{0})\right).
$$
Therefore, $\mathcal{H}$ is a principal connection if and only if $\mathfrak{m}=\mathfrak{g}_{-1}\oplus \mathfrak{z}(\mathfrak{g}_{0})$ is a reductive complement of $\mathfrak{h}$, and we know from Lemma \ref{lemalgebra} that a such complement does not exist.
}
\end{remark}

The distribution $\mathcal{H}$  is a general connection on $p:Q^{\omega}\to N$ \cite[Sect. 1.3.2]{CS09}. 
Therefore, we can consider the horizontal lift $V^{\text{hor}}$ of a vector field $V\in \mathfrak{X}(N)$. 
For each $1\leq i\leq m$, let us consider the constant vector field $\omega^{-1}(E_{i})\in \mathfrak{X}(Q^{\omega})$ for $E_{i}\in \mathfrak{g}_{-1}$ given in (\ref{170320B}).
The distribution $\mathcal{H}$ is globally spanned by the set of constant vector fields $\{\omega^{-1}(E), \omega^{-1}(E_{1}), \cdots ,\omega^{-1}(E_{m})\}$. For every $y\in N$ and  $b=(Z(y),e_{1}, \cdots , e_{m})\in Q^{\omega}$, we have $T_{b}p\cdot \omega^{-1}(E_{j})(b)= e_{j}$ for $1\leq j\leq m$. On the other hand, we know from (\ref{234}) that  $T_{b}p\cdot \omega^{-1}(E)(b)= Z(y)$.

The curvature of the general connection $\mathcal{H}$ is defined by $R(\xi, \eta)=-\mathcal{V}[\mathcal{H}(\xi), \mathcal{H}(\eta)]$ for $\xi, \eta \in \mathfrak{X}(Q^{\omega})$, where $\mathcal{V}$ and $\mathcal{H}$ also denotes the projections according to the decomposition $TQ^{\omega}= \mathcal{V}\oplus \mathcal{H}$ \cite[Sect. 1.3.2]{CS09}. 
From the Frobenious Theorem, the condition $R=0$ codifies the integrability of the distribution $\mathcal{H}$. 
The complete information about $R$ is contained in $$R(\omega^{-1}(E_{i}), \omega^{-1}(E_{j}))=-\mathcal{V}[\omega^{-1}(E_{i}), \omega^{-1}(E_{j})], \quad 0\leq i,j \leq m,$$
where $E_{0}=E$ is the grading element. 

\begin{lemma}\label{160320B}
Let $(p:Q^{\omega}\to N, \omega)$ be a lightlike Cartan geometry. Then, for every $ f\in \mathrm{Iso}(N,h^{\omega}, Z^{\omega})$, the lift $F=Tf: Q^{\omega}\to Q^{\omega}$ preserves the distribution $\mathcal{H}$ if and only if 
\begin{equation}\label{160320A}
TF\cdot \omega^{-1}(E_{j})=\omega^{-1}(E_{j})\circ F, \quad 0\leq j\leq m.
\end{equation}
\end{lemma}
\begin{proof}
Assume $F$ preserves the distribution $\mathcal{H}$. That is, for every $b\in Q^{\omega}$ with $p(b)=y$, there are real numbers $a_{ji}$ such that $$T_{b}F\cdot  \omega^{-1}(E_{j})(b)=\sum_{i=0}^{m}a_{ji}\, \omega^{-1}(E_{j})(F(b)).$$
This implies that $T_{y}f \cdot e_{j}=\sum_{i=0}^{m}a_{ji}\, T_{y}f \cdot e_{j}$ and so $a_{ji}=\delta_{ji}$. Therefore equation (\ref{160320A}) holds.
The other implication is obvious
\end{proof}

\begin{corollary}\label{27121}
Given a lightlike Cartan geometry $(p:Q^{\omega}\to N, \omega)$, we have
$$
\mathrm{Aut}(Q^{\omega},\omega)=\left\{f\in \mathrm{Iso}(N,h^{\omega}, Z^{\omega}): F=Tf \textrm{ preserves the distribution }\mathcal{H}\right\}.
$$
\end{corollary}
\begin{proof}
For every $(Tf, f)\in \mathrm{Aut}(Q^{\omega},\omega)$, we know that $f\in \mathrm{Iso}(N,h^{\omega}, Z^{\omega})$ from (\ref{18121}).
On the other hand, \cite[Lemma 1.5.2]{CS09} implies that $F:=Tf$ preserves all constant vector fields, in particular $TF\cdot \omega^{-1}(Y)=\omega^{-1}(Y)\circ F$  for $Y\in \mathfrak{g}_{-1}\oplus \mathfrak{z}(\mathfrak{g}_{0})$ and $F$ preserves the distribution $\mathcal{H}$.
 Conversely, consider $f\in \mathrm{Iso}(N,h^{\omega},Z^{\omega})$, then it is clear that $F:=Tf:Q^{\omega}\to Q^{\omega}$ and $F(b\cdot \sigma)=F(b)\cdot \sigma$ for all $u\in Q^{\omega}$ and $\sigma\in H$. Hence, $F$ is a $H$-principal fiber bundle morphism. This gives $T_{b}F\cdot \xi_{X}(b)= \xi_{X}(F(b))$ for every $X\in \mathfrak{h}$ and so
$
\omega(F(b))(T_{b}F\cdot \xi_{X}(b))=\omega(b)(\xi_{X}(b))=X.
$
That is, $F$ pulls back $\omega $ to $\omega$ on the vertical distribution of $p:Q^{\omega}\to N$.
Now, the condition $TF\cdot \omega^{-1}(Y)=\omega^{-1}(Y)\circ F$ for every $Y\in \mathfrak{g}_{-1}\oplus \mathfrak{z}(\mathfrak{g}_{0})$ follows from Lemma \ref{160320B}
and thereofre, \cite[Lemma 1.5.2]{CS09} ends the proof.
\end{proof}

\noindent It would be desirable a description of the condition $F$ preserves the distribution $\mathcal{H}$ in terms of the lightlike manifold $(N, h^{\omega}, Z^{\omega})$.

\begin{example}\label{235}
{\rm 
There are lightlike manifolds $(N,h,Z)$ with $\mathrm{Iso}(N,h,Z)$ infinite-dimensional.
The following example is taken from \cite{BFZ}, where the authors include a deep discussion on the isometry group of a lightlike manifold. They call the following one as  {\it the most flexible example}. Let us consider $\R^{m+1}=\R\oplus \R^{m}$ with canonical coordinates $(x_{0}, \cdots , x_{m})$ and endowed with the lightlike metric $h=0\oplus g_{_{\R^{m}}}$.
The vector field $Z=\frac{\partial}{\partial x_{0}}$ spans the radical distribution of $h$.
The group of isometries of $h$ is infinite-dimensional. Even more, we have
$$
\mathrm{Iso}\Big(\R^{m+1},h,\frac{\partial}{\partial x_{0}}\Big)=
$$
$$\Big\{f=(f_{0}, \cdots , f_{m})\in \mathrm{Diff}(\R^{m+1}):\frac{\partial f_{0}}{\partial x_{0}}=1,\,\frac{\partial f_{i}}{\partial x_{0}}=0\,,1\leq i\leq m,\,\tau=(f_{1}, \cdots , f_{m})\in  \mathrm{Euc}(\E^{m})\Big\}.
$$

\smallskip

The lightlike manifold $(\R^{m+1},h,\frac{\partial}{\partial x_{0}})$ can also be constructed from a Cartan lightlike geometry by mean of Theorem \ref{direct} as follows.
The set $Q$ of admissible linear frames for $(\R^{m+1}, h, \frac{\partial}{\partial x_{0}})$ is a trivial principal $H$-fiber bundle on $\R^{m+1}$. Let us fix the following trivialization $Q\cong\R^{m+1}\times H$,
$$
\Big(\frac{\partial }{\partial x_{0}}\mid_{x}, a_{1}\frac{\partial}{\partial x_{0}}\mid_{x}+\sum_{i=1}^{m}g_{i1}\frac{\partial }{\partial x_{i}}\mid_{x}, \cdots ,a_{m}\frac{\partial}{\partial x_{0}}\mid_{x}+\sum_{i=1}^{m}g_{im}\frac{\partial }{\partial x_{i}}\mid_{x}\Big)\mapsto \big(x, [\sigma]=\tiny{\left[\begin{array}{ccc} 1 &  a^{t} & -\frac{1}{2}\| a^{t}\|^2\\ 0 & g & -ga \\ 0 & 0 & 1 \end{array} \right]}\Big),
$$
where $a^{t}=(a_{1}, \cdots, a_{m})$ and $g=(g_{i j})\in O(m)$. Thus, the canonical action on the right of $H\subset Gl(m+1, \R)$ on $Q$ reduces to the trivial action on the right of $H$ on $\R^{m+1}\times H$.

Let $\omega\in  \Omega^{1}(Q,\mathfrak{g})$ be given by
$$
\omega(x, [\sigma])\Big((r,Y)^{t}, \xi_{X}[\sigma]\Big)=\mathrm{Ad}[\sigma^{-1}]\scriptsize{\left(\begin{array}{ccc} r & 0 &  0\\ Y &  0  & 0 \\ 0& -Y^{t}&  -r \end{array} \right) +X,}
$$
where $(r,Y)^{t}\in T_{x}\R^{m+1}\cong \R^{m+1}$, $X\in \mathfrak{h}$ and $[\sigma] \in H$. It follows that $(p:Q\to \R^{m+1}, \omega)$ is a lightlike Cartan geometry on $\R^{m+1}$ with $h^{\omega}=h$ and $Z^{\omega}= \frac{\partial}{\partial x_{0}}$.

We claim that 
$$
\mathrm{Aut}(Q, \omega)=\big\{f\in \mathrm{Diff}(\R^{m+1}):f(x)=\big(x_{0}+c, \tau(x_{1},\cdots, x_{m})\big), \, c\in \R, \tau\in \mathrm{Euc}(\E^{m})\big\}.
$$
Indeed, for every $f\in \mathrm{Iso}(\R^{m+1}, h, \frac{\partial}{\partial x_{0}})$ with $f=(f_{0}, \tau)$  and $\tau=\tiny{\left(\begin{array}{cc} 1 & 0  \\ v &  A\end{array} \right)}\in \mathrm{Euc}(\E^{m})$, a straightforward computation gives
$$
Tf: Q \to Q, \quad (x,[\sigma])\mapsto \Big(f(x),D(x,\tau) [\sigma ]\Big), \textrm{ where }D(x,\tau)=\scriptsize{\left[\begin{array}{ccc} 1 &   (\nabla^{0} f_{0})_{x}^{t} & -\frac{1}{2}\|(\nabla^{0} f_{0})_{x}^{t} \|^2\\ 0 & A  & -A(\nabla^{0} f_{0})_{x} \\ 0 & 0 & 1 \end{array} \right]}
$$
and $(\nabla^{0} f_{0})^{t}=(\frac{\partial f_{0}}{\partial x_{1}},\cdots, \frac{\partial f_{0}}{\partial x_{m}})$.
The map $F=Tf$ preserves the action on the right of $H$ on $Q$. Hence, the condition $(F,f)\in \mathrm{Aut}(Q, \omega)$ reduces to 
$$
\omega\Big(f(x),D(x,\tau)[\sigma] \Big)\big(T_{x}f\cdot (r, Y)^{t}, 0\big)=\mathrm{Ad}[\sigma^{-1}]\scriptsize{\left(\begin{array}{ccc} r & 0 &  0\\ Y &  0  & 0 \\ 0& -Y^{t}&  -r \end{array} \right)}
$$
for all $(r,Y)^{t}\in T_{x}\R^{m+1}$ and $[\sigma]\in H$.
In other words, we have 
$$
\mathrm{Ad}(D(x,\tau))\scriptsize{\left(\begin{array}{ccc} r & 0 &  0\\ Y &  0  & 0 \\ 0& -Y^{t}&  -r \end{array} \right)}=\scriptsize{\left(\begin{array}{ccc} r+\sum_{i=1}^{m}Y_{i}\frac{\partial f_{0}}{\partial x_{i}}\mid_{x} & 0 &  0\\ AY &  0  & 0 \\ 0& -Y^{t}A^{t}&  -r -\sum_{i=1}^{m}Y_{i}\frac{\partial f_{0}}{\partial x_{i}}\mid_{x}\end{array} \right)},
$$
for all $(r,Y)^{t}\in T_{x}\R^{m+1}$ and $[\sigma]\in H$.
This equation holds if and only if $\nabla^{0} f_{0}=0$, that is, there is $c\in \R$ with $f_{0}(x)=x_{0}+c$.


%
%
%
%
%

}
\end{example}

\begin{remark}\label{15121}
{\rm According to \cite[Theorem 5.1]{Kob}, for every $G$-structure $P\subset L(N)$ on a connected $n$-dimensional manifold $M$ with $\mathfrak{g}\subset \mathfrak{gl}(n , \R)$ of finite type, the group $\mathrm{Aut}(P)$ of automorphisms of $P$ is a Lie group. 
Given a lightlike manifold $(N,h,Z)$, the corresponding $H$-principal fiber bundle $Q\subset L(N)$ given in (\ref{laQ}) satisfies $ \mathrm{Aut}(Q)=\mathrm{Iso}(N,h,Z)$. Hence, from Lemma \ref{lemalgebra} $iv)$, the existence of examples with $
\mathrm{Iso}(N,h,Z)$ infinite dimensional was expected.
%
%

}
\end{remark}

\section{Lightlike hypersurfaces}

The following two sections show methods to construct lightlike Cartan geometries from given lightlike manifolds. 
This section deals with the case of lightlike hypersurfaces.

The lightlike manifolds naturally appear as hypersurfaces of Lorentzian manifolds. That is, we have an immersion $\psi:(N^{m+1},h)\to (M^{m+2},g)$ where $(M,g)$ is a Lorentzian manifold and the lightlike metric $h=\psi^{*}(g)$. 
The Levi-Civita connection $\nabla^{g}$ of $(M,g)$ gives rise to the induced connection $\overline{\nabla}^{g}:\mathfrak{X}(N)\times \overline{\mathfrak{X}}(N)\to \overline{\mathfrak{X}}(N),$ where $\overline{\mathfrak{X}}(M)$
is the set of vector fields along the immersion $\psi$, \cite[Chap. 4]{One83}. 
Assume there is a vector field $Z\in \mathfrak{X}(N)$ which globally spans the radical distribution $\mathrm{Rad}(h)$. 
Following \cite{Galloway2}, the integral curves of the vector field $Z$, when suitably parameterized, are lightlike geodesics in $M$.
That is, there is a smooth function $\lambda\in C^{\infty}(N)$ such that 
\begin{equation}\label{231}
\overline{\nabla}^{g}_{Z}Z=\lambda Z.
\end{equation}
The function $\lambda$ is called the expansion function of $(N,h,Z))$ in $(M,g)$.

In \cite{Galloway}, the null-Weingarten map with respect to $Z$ is defined as the linear map  
$$[\overline{\nabla}^{g}Z]:T_{y}N/\mathrm{Rad}(h_{y})\to T_{y}N/\mathrm{Rad}(h_{y}), \quad [v]\mapsto [\overline{\nabla}^{g}_{v}Z].$$
Moreover, $[\overline{\nabla}^{g}Z]:\mathcal{E} \to \mathcal{E}$ is self-adjoint with respect to the bundle-like Riemannian metric 
$\overline{h}$. The null second fundamental form $B_{Z}$ is defined by 
\begin{equation}\label{10121}
B_{Z}([u], [v])=\overline{h}([\overline{\nabla}^{g}Z][u], [v]).
\end{equation}
Since $[\overline{\nabla}^{g}Z]$ is self-adjoint, $B_{Z}$ is symmetric \cite{Galloway}, \cite{Galloway2}. The lightlike hypersurface is said to be totally umbilical when there is a smooth function $\rho$ on $N$ such that $B_{Z}=\rho \,\bar{h}$ and properly totally umbilical when the function $\rho$ is nowhere vanishing. The geometric meaning of totally umbilicity is analyzed in \cite{Perlick}.

For lightlike hypersurfaces, the endomorphisms field $A_{Z}$ on $\mathcal{E}$ determined in (\ref{211}) satisfies 
\begin{equation}\label{240120}
\mathcal{L}_{Z} \overline{h}([u],[v])=2\overline{h}(A_{Z}[u],[v])=2B_{Z}([u],[v]),
\end{equation}
and we get
$
A_{Z}=[\overline{\nabla}^{g}Z].
$
Therefore, the lightlike hypersurface $(N,h,Z)$ is generic if and only if the null-Weingarten map 
$
[\overline{\nabla}^{g}Z]: \mathcal{E}\to \mathcal{E}
$
 is a vector fiber bundle isomorphism. Whereas $\overline{\nabla}^{g}Z: TN\to TN$ is a  vector bundle isomorphism if and only if to the generic condition we add the expansion function $\lambda$ is nowhere vanishing.

\smallskip
Assume $(M^{m+1}, g)$ is a timelike oriented Lorentzian manifold and fix a globally defined timelike vector field ${\bf T}\in \mathfrak{X}(M)$.
Let $\psi:(N^{m+1},h)\to (M^{m+2}, g)$ be a lightlike hypersurface. In this case, there always exists $Z\in \mathfrak{X}(N)$ which spans the distribution $\mathrm{Rad}(h)$.

\begin{definition}
A lightlike hypersurface $\psi:(N^{m+1},h, Z)\to (M^{m+2}, g)$ as above is said to be causally oriented when the vector field $Z\in \mathfrak{X}(N)$ satisfies at all $y\in N$
 $$g\left(T_{y}\psi \cdot Z(y),{\bf T}(\psi(y))\right)<0.$$
\end{definition}

\noindent 
There is no loss of generality in assuming $\psi:(N^{m+1},h, Z)\to (M^{m+2}, g)$ is a causally oriented lightlike hypersurface.

\smallskip

Let $O^{+}(M)$ be the set of all time-oriented $g$-admissible linear frames on $M$. That is, $u\in O^{+}(M)$ whenever
$$u=(\ell^{+}, w_{1}, \cdots , w_{m}, \ell^{-})\in L_{x}(M), \quad x\in M,$$
where for all $1\leq i,j \leq m$, we have $g(w_{i}, w_{j})= \delta_{ij}$, $g(w_{i}, \ell^{+})=g(w_{i}, \ell^{-})=g(\ell^{+}, \ell^{+})= g(\ell^{-}, \ell^{-})=0$, $g(\ell^{+}, \ell^{-})=1$ and $\ell^{+}$ is future pointed, that is $g(\ell^{+}, {\bf T}(x))<0$.

Let $O^{+}(m+1,1)\subset O(m+1,1)$ be the subgroup of the Lorentz group preserving the future cone $\mathcal{N}^{m+1}$. That is, $\sigma=(v\vert e_{1}\vert\cdots \vert e_{m}\vert y )\in O^{+}(m+1,1)$ if and only if $v\in \mathcal{N}^{m+1}$.
The M\"obius Lie group $G=PO(m+1,1)$ can be identified with $O^{+}(m+1,1)\subset Gl(m+2, \R)$ by means of
$$
G\to O^{+}(m+1,1), \quad [\sigma] \mapsto \sigma^{+},
$$
where $\sigma^{+}\in [\sigma]$ is the unique representative with $\sigma^{+}\in O^{+}(m+1,1)$.
In this way, the M\"obius Lie group $G=PO(m+1,1)$
acts on the right on $O^{+}(M)\subset L(M)$. The natural projection $\pi:O^{+}(M)\to  M$ is a principal fiber bundle with structure group $G$
and $O^{+}(M)$ is a reduction of the structure group $Gl(m+1, \R)$ on $L(M)$ to $G$.

Given $Q\subset L(N)$, the set of all admissible linear frames of $N$ in (\ref{laQ}), the immersion $\psi$ is lifted to $\Psi:=T\psi:Q\to O^{+}(M)$ as follows 
\begin{equation}\label{23121}
\Psi(b)=\Psi \left(Z(y),e_{1}, \cdots , e_{m}\right)=\big(T_{y}\psi\cdot Z(y),T_{y}\psi\cdot e_{1}, \cdots , T_{y}\psi\cdot e_{m}, \eta(b)\big)
\end{equation}
where $\eta(b)\in T_{\psi(y)}M$ is determined by the condition $\Psi(b)\in O^{+}(M)$. In this way, we get a map from $Q$ to the lightlike vectors of $M$
\begin{equation}\label{1501201}
\eta:Q\to TM,\quad  b\mapsto \eta(b).
\end{equation}
It is a straightforward computation to show that $(\Psi, \psi)$ is an homomorphism of principal fiber bundle with group homomorphism the inclusion $H\subset G$. That is, $\Phi( b\cdot h )=\Phi(b)\cdot h 
$ for all $h\in H, b\in Q$ and the following diagram commutes.
$$
\begin{CD}
Q  @>\Psi >> O^{+}(M)\\
@V p VV @VV \pi V \\
N @>> \psi > M
\end{CD}
$$

\noindent A standard argument for immersions shows.
\begin{lemma}\label{tecnico}
For every $b\in Q$ with $p(b)=y$, there are local sections $s:W\to Q$ and $\tau:U\to O^{+}(M)$ such that $s(y)=b$,
$\psi(W)\subset U$ and $\Psi\circ s =\tau\circ \psi$ on $W$.
\end{lemma}
\begin{proof} We can pick $\left(U, \phi=(r_{0},r_{1}, \cdots , r_{m}, r_{m+1})\right)$ a cubic coordinate system centered at $\psi(y_{0})\in M$ such that $\psi(N)\cap U$ consists only of the slice $r_{m+1}=0$. Let $s$ be a local section of $Q$ on an open set $y\in W\subset N$ with $\psi \mid_{W}$ one-to-one and $s(y)=b$. 
Assume $U$ has been chosen small enough so that $\psi(W)=\psi(N)\cap U$ and there exists a local section $\hat{\tau}:U\to O^{+}(M)$.
Let us consider the map $a:W \to G$ determined by the equation $\Psi \circ s = r^{a}\circ \hat{\tau}\circ \psi$. The map $a$ can be extended to all $U$ as follows
$$
\hat{a}:U\to G, \quad \hat{a}(x)=(a\circ \psi^{-1})\Big[\phi^{-1}\Big(r_{0}(x), r_{1}(x),\cdots , r_{m}(x), 0\Big)\Big].
$$
Then, the section $\tau:=r^{\hat{a}}\circ \hat{\tau}$ on $U$ has the stated properties.
\end{proof}

Let $\gamma^{g}\in \Omega^{1}(O^{+}(M), \mathfrak{g})$ be the (Ehresmann) principal connection form corresponding to the Levi-Civita connection $\nabla^g$. Recall that $\gamma^{g}$ is given in terms of the covariant differentation $\nabla^{g}$ as follows. For every $u\in O^{+}(M)$ and $V\in T_{u}O^{+}(M)$, take a local section $\tau=(\ell^{+}, E_{1}, \cdots , E_{m},\ell^{-} )$ on an open set $U\subset M$ around $x=\pi(u)$ with $\tau(x)=u$. We have the decomposition
$
T_{u}O^{+}(M)=\mathrm{Im}(T_{x}\tau)\oplus \mathcal{V}(u)
$
where $\mathcal{V}(u)$ denotes the vertical distribution of the fiber bundle $\pi:O^{+}(M)\to M$. Then, writing $V=T_{x}\tau\cdot w+\xi_{Y}(u)$ with $Y\in \mathfrak{o}(m+1,1)$, we have 
\begin{equation}\label{221}
\gamma^{g}(u)(V)=\left(\begin{array}{ccc} g(\nabla^{g}_{w}\ell^{+}, \ell^{-})  & Z & 0\\ X & A & -Z^{t} \\ 0 & -X^{t}& -g(\nabla^{g}_{w}\ell^{+}, \ell^{-}) \end{array} \right)+ Y\in \mathfrak{g},
\end{equation}
where $X=(g(\nabla^{g}_{w}\ell^{+} ,E_{1}), \cdots , g(\nabla^{g}_{w}\ell^{+} ,E_{m}))^{t}$, $Z=(g(\nabla^{g}_{w}E_{1}, \ell^{-}), \cdots , g(\nabla^{g}_{w}E_{m}, \ell^{-}))$
and $A=\Big(g(E_{i}, \nabla^{g}_{w}E_{j})\Big)_{1\leq i, \, j\leq m}\in \mathfrak{o}(m)$.

\begin{lemma}\label{232}
Let $\omega:= \Psi^{*}(\gamma^g)\in \Omega (Q, \mathfrak{g})$ be
the pull back of the Levi-Civita connection form. For every $b=(Z(y), e_{1}, \cdots, e_{m})\in Q$ and $\xi\in T_{b}Q$, let $s$ be any local section of $p:Q\to N$ around $y\in N$ with $s(y)=b$. Then, we have
$$
[\omega(b)(\xi)]_{\mathfrak{z}(\mathfrak{g}_{0})}=g(T_{y}\psi\cdot \overline{\nabla}^{g}_{v}Z, \eta(b))\,\,\, \text{and}\,\,\,
[\omega(b)(\xi)]_{-1}=\Big(h\Big( \overline{\nabla}^{g}_{v}Z, e_{1}\Big),\cdots , h\Big(\overline{\nabla}^{g}_{v}Z, e_{m}\Big) \Big)^{t},
$$
where  $\xi=T_{y}s\cdot v +\xi_{Y}(b)\in T_{b}Q$, $Y\in \mathfrak{h}$,
and the subscripts denote the projections from $\mathfrak{g}$ on $\mathfrak{z}(\mathfrak{g}_{0})$ and $\mathfrak{g}_{-1}$, respectively. 
\end{lemma}
\begin{proof}
Assume local sections $s=(Z,F_{1}, \cdots , F_{m})$ and $\tau=(\ell^{+},E_{1}, \cdots , E_{m},\ell^{-})$ as in Lemma \ref{tecnico}.
Then, we have $\ell^{+}\circ \psi=T\psi \circ Z$ and $\ell^{-}\circ \psi=\eta \circ s$ on $W\subset N$ where the map $\eta$ is given in (\ref{1501201}). Now, the proof is a straightforward computation from (\ref{221}).
%
%
%
%
%
\end{proof}

\begin{theorem}\label{main1} Let $\psi:(N,h, Z)\to (M,g)$ be a causally oriented lightlike hypersurface of a Lorentzian manifold $(M,g)$.
Then $\omega:=\Psi^{*}(\gamma^{g})\in \Omega^{1}(Q, \mathfrak{g})$ is a Cartan connection if and only if $\overline{\nabla}^{g}Z$ is a vector fiber bundle isomorphism. That is, $(N,h,Z)$ is generic and the expansion function $\lambda$ given in (\ref{231}) is nowhere vanishing. 
In this case, the lightlike metric $h^{\omega}$ and the vector field $Z^{\omega}$ deduced from $\omega$ by means of Theorem \ref{direct} satisfy
$$
h^{\omega}_{_{[\overline{\nabla}^{g}Z]^{-1}}}=h\quad\,\mathrm{and } \quad \, Z^{\omega}=\frac{1}{\lambda}Z,
$$
where $h^{\omega}_{_{[\overline{\nabla}^{g}Z]^{-1}}}$ is given in (\ref{210120}).
\end{theorem}
\begin{proof} 
Taking into account that $\gamma^{g}$ is a principal connection and $(\Psi, \psi)$ a principal fiber bundle morphism, one easily shows that 
$\omega$ is $H$-equivariant and reproduces the generators of the fundamental vector fields. That is, items $b)$ and $c)$ in the definition of a Cartan connection. Therefore, the only non-trivial assertion is the equivalence between the regularities of
$\overline{\nabla}^{g}Z$ and $\omega(b):T_{b}Q\to \mathfrak{g}$ for every $b\in Q$. We have that $\omega$ reproduces the generators of the fundamental vector fields and therefore $  \mathrm{Im}(\omega(b))  \supseteq \mathfrak{h}$.
In light of this argument, one has to check that
$\overline{\nabla}^{g}Z$ is a vector bundle isomorphism if and only if
$\mathfrak{g}_{-1}\oplus \mathfrak{z}(\mathfrak{g}_{0})\subset \mathrm{Im}(\omega(b))$.
Let $s$ be a local section of $p:Q\to N$ around $y\in N$ with $s(y)=b=(Z(y), e_{1}, \cdots , e_{m})$. Then, Lemma \ref{232} gives
\begin{equation}\label{23123}
\omega(b)(T_{y}s\cdot Z(y))=\lambda(y)E\neq 0
\end{equation}
and so $\mathfrak{z}(\mathfrak{g}_{0})\subset \mathrm{Im}(\omega(b))$.
Moreover, for every $v\in  T_{y}N$, we also have from Lemma \ref{232}
\begin{equation}\label{23122}
[\omega(b)(T_{y}s\cdot v)]_{-1}=\Big(h\Big( \overline{\nabla}^{g}_{v}Z, e_{1}\Big),\cdots , h\Big(\overline{\nabla}^{g}_{v}Z, e_{m}\Big) \Big)^{t}.\end{equation}
Thus, it is easily seen that the generic condition implies $ \mathfrak{g}_{-1}\subset \mathrm{Im}(\omega(b))$.
The proof for the converse is similar and it also follows from Lemma \ref{232}.  
%

Assume now $(p:Q\to N, \omega=\Psi^{*}(\gamma^{g}))$ is a lightlike Cartan geometry on $N$. From (\ref{233}) and (\ref{23122}), we get for the lightlike metric $h^{\omega}$,
\begin{equation}\label{237}
h^{\omega}(u,v)=\sum_{i=1}^{m}h(\overline{\nabla}^{g}_{u}Z, e_{i}) h(\overline{\nabla}^{g}_{v}Z, e_{i})=h(\overline{\nabla}^{g}_{u}Z,\overline{\nabla}^{g}_{v}Z), \quad u,v\in T_{y}N,
\end{equation}
which implies $h^{\omega}_{_{[\overline{\nabla}^{g}Z]^{-1}}}=h$,
Finally, from (\ref{234}) and (\ref{23123}), we have $Z^{\omega}(y)=T_{b}p\cdot \omega(b)^{-1}(E)=\frac{1}{\lambda(y)}Z(y).$
\end{proof}

\begin{remark}
{\rm Under the assumption of Theorem \ref{main1}, there are two different lightlike metrics $h$ and $h^{\omega}$ on $N$, in general. 
The lightlike metrics $h$ and $h^{\omega}$ agree if and only if $[\overline{\nabla}^{g}Z]:\mathcal{E}\to \mathcal{E}$ is a $\overline{h}$-isometry but
$h$ and $h^{\omega}$ share the radical distribution. 
The principal fiber bundle $Q^{\omega}$ given in (\ref{236}) may also be different from $Q$. In spite of this situation, $Q^{\omega}$ and $Q$ are isomorphic as $H$-principal fiber bundles on $N$. In fact, it follows from (\ref{237}) that $\overline{\nabla}^{g}Z(Q^{\omega})=Q$, where $\overline{\nabla}^{g}Z(u):=(Z(y),\overline{\nabla}^{g}_{u_{1}}Z, \cdots ,\overline{\nabla}^{g}_{u_{1}}Z)\in Q$ for every $u=(Z^{\omega}(y), u_{1}, \cdots , u_{m})\in Q^{\omega}$.
These facts are summarized as follows
$$
(N,h,Z)\stackrel{Th. \ref{main1}}{\longrightarrow}\Big(p:Q\to N, \omega=\Psi^{*}(\gamma^{g})\Big) \stackrel{Th. \ref{direct}}{\longrightarrow} \Big(N, h^{\omega}, Z^{\omega}=\frac{1}{\lambda}Z\Big).
$$
}
\end{remark}

\begin{corollary}\label{umbilical1}
Let $\psi:(N,h, Z)\to (M,g)$ be a totally umbilical causally oriented lightlike hypersurface of a Lorentzian manifold $(M,g)$ with $B_{Z}=\rho \, \bar{h}$.
Then $\omega=\Psi^{*}(\gamma^{g})\in \Omega^{1}(Q, \mathfrak{g})$ is a Cartan connection if and only if the functions $\rho$ and the expansion $\lambda$ are nowhere vanishing. In this case, the lightlike metric $h^{\omega}$ and the vector field $Z^{\omega}$ deduced from $\omega$ are
$h^{\omega}=\rho^{2}\, h$ and  $Z^{\omega}=\frac{1}{\lambda}Z$, respectively. 
\end{corollary}
%
%
%
\begin{example}\label{1701201}
{\rm Lorentzian manifolds $(M,g)$, with a recurrent lightlike vector field $Z$, are considered in \cite{Leistner}.  In this case, there is a $1$-form $\alpha$ on $M$ such that $\nabla^{g}Z=\alpha \otimes Z$. The manifold $M$ is foliated by the lightlike hypersurfaces $Z^{\perp}$. For every $f\in C^{\infty}(M)$, let us consider the conformally related metric $\widehat{g}=e^{2f}g$. The induced connection $\overline{\nabla}^{\widehat{g}}$ on each lightlike hypersurface $Z^{\perp}$ 
satisfies
$$
\overline{\nabla}_{V}^{\widehat{g}}Z=\alpha(V) Z+V(f)Z+Z(f)V,\quad V\in \mathfrak{X}(Z^{\perp}).
$$
The null Weingarten map is $[\overline{\nabla}^{\widehat{g}}Z]=Z(f)\cdot \mathrm{Id}_{\mathcal{E}}$ and the expansion function $\lambda=\alpha(Z)+ 2 Z(f)$. Thus, Corollary \ref{umbilical1} applies to the every  lightlike hypersurface $Z^{\perp}$ where $Z(f)$ and $\alpha(Z)+ 2 Z(f)$ are never vanishing functions.
}
\end{example}

\begin{example}\label{10123}
{\rm The notion of {\it lightlike hypersurface $N$ with base $M^2$} in $\L^{4}$ was introduced in \cite{Koss} as a smooth map $i:M\times \R= N \to \L^{4}$ (nonimmersive on a set $S(i)$) such that
\begin{enumerate}
\item For all $y\in N-S(i)$ the induced metric $i^{*}\langle \,,\, \rangle$ is lightlike. 
\item The fibers of $\pi :N-S(i) \to M$ are lightlike geodesics.
\item There exists a  smooth spacelike section $r$ of $\pi :N \to M$ which does not intersect $S(i)$.
\end{enumerate}
Let us consider the upper half lightlike hypersurface $(N^{+}- S(i), h:=i^{*}(g), Z=t\partial_{t})$ with $N^{+}=M\times \R_{>0}$.
Given a frame $(Z(y), v_{1},v_{2})$ of $T_{y}N$, Kossowski sets the following curvature quantity  \cite[Def. 5]{Koss} (a normalization condition on $Z$ is added in \cite{Koss}) 
$$\overline{K}(y):= \frac{\mathrm{det}\,h(\overline{\nabla}^{0}_{v_{i}}Z,v_{j})}{\mathrm{det}\,h(v_{i}, v_{j})}.$$ 
The generic condition is satisfied if and only if $\overline{K}$ is a nowhere vanishing function.
Moreover, the above condition $2)$ on the fibers of $\pi$ is equivalent to $\overline{\nabla}^{0}_{\partial_{t}}\partial_{t}=0$ and hence $\lambda=1$. 
The notion of lightlike hypersurface $N$ with base $M$ can be extended to general ambient Lorentzian manifolds in a natural way.
}
\end{example}

Recall that the curvature of the principal connection form $\gamma^{g}\in \Omega^{1}(O^{+}(M), \mathfrak{g})$ is defined by
$
\Omega=d\gamma^{g}+1/2[\gamma^{g}, \gamma^{g}]\in \Omega^{2}(O^{+}(M), \mathfrak{g}).
$
For tangent vectors $v,w\in T_{x}M$ and $u\in O^{+}(M)$ with $\pi(u)=x$, the Riemann curvature tensor $R^{g}$ is determined by
$$
R^{g}(v,w)=u\Big[\Omega(v^{*}, w^{*})\circ u^{-1}\Big],
$$
where $v^{*}, w^{*}\in T_{u}O^{+}(M)$ with $T_{u}\pi\cdot v^{*}=v$ and $T_{u}\pi\cdot w^{*}=w$. Here $\Omega(v^{*}, w^{*})\in \mathfrak{g}$ denotes the corresponding endomorphism of $\R^{m+2}$ and $u:\R^{m+2}\to T_{x}M$ the linear isomorphism deduced from $u\in O^{+}(M)$, \cite[Chap. 3]{KN63}. Now, let $\psi:(N,h, Z)\to (M,g)$ be a causally oriented lightlike hypersurface 
 with $\overline{\nabla}^{g}Z$ a vector fiber bundle isomorphism. 
 The curvature form $K^{\omega}$ of the Cartan connection $\omega$ in Theorem \ref{main1} satisfies
$
K^{\omega}=\Psi^{*}(\Omega).
$
Therefore, for $\xi_{b}, \zeta_{b}\in T_{b}Q$ with $p(b)=y$, we get
\begin{equation}\label{26121}
K^{\omega}(\xi_{b}, \zeta_{b})=\Omega(T_{b}\Psi\cdot \xi_{b}, T_{b}\Psi\cdot \zeta_{b})=
\Psi(b)^{-1}\Big[R^{g}\Big(T_{b}(\psi \circ p)\cdot \xi_{b}, T_{b}(\psi \circ p)\cdot \zeta_{b}\Big)\circ \Psi(b)\Big].
\end{equation}
As a direct consequence from (\ref{26121}), we have.
\begin{enumerate}
\item The curvature form $K^{\omega}=0$ if and only $R^{g}$ vanishes identically on $\psi(N)$. That is, $R^{g}(T_{y}\psi\cdot v,T_{y}\psi\cdot w)=0$ on $T_{\psi(y)}M$ for all $v,w\in T_{y}N$.

\item For $E=\tiny{\left(\begin{array}{ccc} 1 &  0 & 0\\ 0 & 0 & 0 \\ 0 & 0 & -1 \end{array} \right)}\in \mathfrak{g}$, the curvature function $k^{\omega}:Q\to \Lambda^{2}(\mathfrak{g}/\mathfrak{h})^{*}\otimes \mathfrak{g}$ satisfies
$$
k^{\omega}(b)(E+\mathfrak{h}, Y+\mathfrak{h})= \Psi(b)^{-1}\Big[\frac{1}{\lambda(y)}R^{g}\Big(T_{y}\psi\cdot Z(y), T_{y}\psi\cdot \omega^{-1}(Y)(y)\Big)\circ \Psi(b)\Big].
$$
Hence, $k^{\omega}(b)(E+\mathfrak{h}, *)=0$ if and only if $R^{g}(T_{y}\psi\cdot Z(y), T_{y}\psi\cdot v)=0$
on $T_{\psi(y)}M$ for all $v\in T_{y}N$. 


\end{enumerate}
Combining these two vanishing results with general characterizations of homogeneous models \cite[Prop. 1.5.2]{CS09} and correspondence spaces \cite[Theor. 1.5.14]{CS09}, we obtain. 

\begin{corollary}\label{13121}
 Let $\psi:(N^{m+1},h, Z)\to (M,g)$ be causally oriented lightlike hypersurface of a Lorentzian manifold $(M,g)$. Assume $(N,h,Z)$ is generic and the expansion function $\lambda$ is nowhere vanishing. Then,
 \begin{enumerate}
 \item $(N, h^{\omega}, \lambda^{-1}Z)$ is locally isomorphic, as lightlike Cartan geometry, to the future lightlike cone $\mathcal{N}^{m+1}\subset \L^{m+2}$ if and only if $R^{g}$ vanishes identically on $\psi(N)$. In this case, the lightlike metric $h^{\omega}=\lambda^{2}h$.
 \item For every $y \in N$ there is a small enough local twistor space $\pi:U\to M$ and a section $\tau:M \to U$ such that $(U, h^{\omega},  \lambda^{-1}Z)$ is isomorphic, as lightlike Cartan geometry, to the fiber bundle of scales of the conformal Riemannian manifold $(M, [\tau^{*}(h^{\omega})])$
 if and only if $R^{g}(T\psi\cdot Z, T\psi\cdot V)=0$ for every $V\in \mathfrak{X}(N)$ (see Remark \ref{292}). In this case, the lightlike metric $h^{\omega}=\lambda^{2}h$.
 \end{enumerate}
\end{corollary}
\begin{proof}
Only the assertion $h^{\omega}=\lambda^{2}h$ is not a direct application of \cite[Prop. 1.5.2]{CS09} and \cite[Theor. 1.5.14]{CS09} from the above items $1)$ and $2)$, respectively.
In both cases, $[A_{Z^{\omega}}]$ is the identity map, hence, for the vector field $Z^{\omega}=\frac{1}{\lambda}Z$ we get
$$
[v]=[A_{Z^{\omega}}(v)]=\Big[\overline{\nabla}^{g}_{v}\left(\frac{1}{\lambda}Z\right)\Big]= \frac{1}{\lambda} [\overline{\nabla}^{g}_{v }Z], \quad v\in T_{y}N,
$$
and therefore $h^{\omega}=\lambda^{2}h$.
%
%
\end{proof}

\begin{corollary}
Let $\psi:(N^{m+1},h, Z)\to \L^{m+2}$ be causally oriented lightlike hypersurface. Assume $(N,h,Z)$ is generic and the function $\lambda$ is nowhere vanishing. Then, $(N,\lambda^{2}\, h, \lambda^{-1}Z)$ is locally isomorphic, as lightlike Cartan geometry, to the future lightlike cone $\mathcal{N}^{m+1}\subset \L^{m+2}$.
\end{corollary}

\begin{remark}
{\rm The assumptions of Corollary \ref{13121} imply $[A_{Z^{\omega}}]=\mathrm{Id}_{\mathcal{E}}$. Hence, from Remark \ref{291}, there exist local coordinates $(s, r_{1}, \cdots , r_{m})$ in which the lightike metric has the following local shape $$h= \exp(2s)\sum_{1\leq i,j, \leq m}C_{ij }(r_{1}, \cdots , r_{m})dr_{i}\otimes dr_{j}.$$
}
\end{remark}

\section{Ambient Lorentzian metrics for lightlike manifolds}

This section deals with a construction of an ambient Lorentzian metric $(M,g)$ for certain lightlike manifolds $(N^{m+1},h,Z)$. 
That is, a Lorentzian manifold $(M,g)$, which admits $(N,h,Z)$ as lightlike hypersurface, with expansion function $\lambda=1$ and such that the generic condition will permit to assure that $\overline{\nabla}^{g}Z$ is a vector bundle isomorphism and, therefore, Theorem \ref{main1} can be applied.

Let  $(N^{m+1},h,Z)$ be a lightlike manifold.
To star with, let us note that, without loss of generality, the vector field $Z$ can be assumed to be complete. In fact, we endow $N$ with a complete Riemannian metric $g_{_{R}}$. This metric $g_{_{R}}$ always exists due to a theorem by Nomizu and Ozequi \cite{NoOz}. Then, we replace $Z$ by the vector field $Z/\| Z\|_{g_{_{R}}}$ which is complete. Let us recall that the genericity condition on $(N^{m+1},h,Z)$ is independent of the choice of $Z$ orienting $\mathrm{Rad}(h)$ \cite{BZ17}.
Thus, we get an action on the right $N\times \R \to N$
given by the flows $y \mapsto \mathrm{Fl}^{Z}_{t}(y)$.

Following the guideline example (section 2), we prefer a multiplicative action on the right as follows
\begin{equation}\label{301}
N\times \R_{>0}\to N, \quad (y, s)\mapsto   \mathrm{Fl}^{Z}_{\log s}(y):=y\cdot s.
\end{equation}
Let us assume this action is proper and free. In this case, the orbit space $M:=N/\R_{>0}$ is an $m$-dimensional smooth manifold and the natural projection $\pi: N\to M$ is a trivial principal fiber bundle with structure group $\R_{>0}$ (see for instance \cite[App. E]{Sharpe} or \cite[Theor. 2.9]{Flores}). 
The manifold $M=N/\R_{>0}$ is physically interpreted as the surface of the light source and $\pi: N\to M$ is a global twistor space for $(N,h,Z)$ (Definition \ref{82}).
In this situation, the manifold $N$ is a generalized conformal structure on $M$ in the sense of \cite{BZ17} as follows,
$$
N\to \mathrm{Sym}^{+}(TM), \quad y \mapsto g_{y},
$$
here $\mathrm{Sym}^{+}(TM)$ denotes the positive definite symmetric bilinear forms on $M$ and $g_{y}$ is the positive definite scalar product on $T_{x}M$ obtained from the projection
$T_{y}\pi:T_{y}N\to T_{x}M$. Every section $\Gamma:M\to N$ of $\pi$ gives a Riemannian metric on $M$ by means of $\Gamma^{*}(h)$. In general, there is no relation between Riemannian metrics on $M$ obtained from different sections of $\pi:N\to M$.

\begin{remark}\label{11121}
{\rm For lightlike hypersurfaces $(N,h,Z)$ of a Lorentzian manifold $(M,g)$, the conditions proper and free for the action (\ref{301}) are closely related with causality conditions on the ambient manifold. For example, if $(M,g)$ is globally hyperbolic, $N$ is embeded  in $M$ and the vector field $Z$ can be extended to $M$ as a causal vector field, then it is not difficult to get  from \cite[Cor. 3.14]{Flores} that the action (\ref{301}) is proper and free.
}
\end{remark}
\noindent Every (spacelike) section $\Gamma:M\to N$ provides a diffeomorphism, that is, a global trivialization of $\pi: N\to M$ as follows
\begin{equation}\label{diff}
M\times \R_{>o}\to N, \quad (x,s)\mapsto \mathrm{Fl}_{\log s}^{Z}(\Gamma(x))=\Gamma(x)\cdot s.
\end{equation}
Thus, $M$ is endowed with a $1$-parametric family of Riemannian metrics $g_{s}:=(\Gamma_{s})^{*}(h)$ induced from the sections $\Gamma_{s}(x):= \Gamma (x)\cdot s$ for all $s>0$.
Under the diffeomorphism (\ref{diff}), the lightlike metric $h$ and the vector field $Z$ read as $q\vert_{(x,s)}=g_{s}\oplus 0$ and $s\partial_{s}$, respectively (compare with Remark \ref{292}).

In a similar way to the Lorentzian metric given in (\ref{29122}) for the guideline example $\mathcal{N}^{m+1}$, we will construct an ambient Lorentzian metric for $(N,h,Z)$ when the vector field $Z$ is complete and the action (\ref{301}) is proper and free. This construction is inspired by the ambient metric construction by Feffermann and Graham \cite{F-G}. To start with, let us fix a section $\Gamma:M \to N$ and the corresponding trivialization (\ref{diff}). On the product manifold ${\bf M}:=(-\varepsilon, +\varepsilon)\times M\times \R_{>0}$ (for small enough $\varepsilon >0$), we define the
family of {\it ambient} Lorentzian metrics
\begin{equation}\label{29123}
g^{\sigma}\mid_{(\rho, x,s)}=ds\otimes d(\rho s)+d(\rho s)\otimes ds +\sigma^{2}(\rho)g_{s},
\end{equation}
where $\sigma\in C^{\infty}(-\varepsilon, +\varepsilon)$ with $\sigma >0$ and $\sigma(0)=1$. The function $\sigma$ will be determined later. The natural embedding of $M\times \R_{>0}$ in ${\bf M}$ at $\rho=0$ recovers the original lightlike metric $q$. All metrics $g^{\sigma}$ satisfy $g^{\sigma}(\partial_{s}, \partial_{s})=2\rho$, $g^{\sigma}(\partial_{\rho}, \partial_{\rho})=0$ and $g^{\sigma}(\partial_{s}, \partial_{\rho})=s$. The causal character of the hypersurfaces $M\times \R_{>0}$ at constant $\rho_{0}$ depends on the sign of $\rho_{0}$. These hypersurfaces are Lorentzian for $\rho_{0}< 0$ and Riemannian for $\rho_{0}>0$. 
All metrics $g^{\sigma}$ agree on the the band $B=(-\varepsilon, +\varepsilon)\times \R_{>0}$ and we write
$g^{\sigma}\vert_{B}:=g_{_{B}}$ omitting the superscript $\sigma$.

\begin{remark}\label{warped}
{\rm The metrics $g^{\sigma}$ in (\ref{29123}) are not warped product metrics $B\times_{f} M$, \cite[Chap. 7]{One83}, in general. Even more, we have $g^{\sigma}=g_{_{B}}+ f^{2}(\rho, s)g_{1}$ for some $f >0$ if and only if there is a smooth function $u(s)$ such that $g_{s}= e^{2u(s)}g_{1}$. 
Of course, $g_{s}= e^{2u(s)}g_{1}$ implies $g^{\sigma}$ is a warped product metric. Conversely, assume $\sigma^{2}(\rho)g_{s}=f^{2}(\rho, s)g_{1}$. Take any $g_{s}$-unitary tangent vector $v_{s}$. Then we have $f(\rho, s)=\sigma(\rho)/\|v_{s}\|_{1}$. Thus, the value $\epsilon (s)=1/\|v_{s}\|_{1}$ depends only on $s$. Hence, $f(\rho, s)=\epsilon (s)\sigma(\rho)$ and we get $g_{s}=\epsilon^{2} (s)g_{1}$. In this case, the manifold $N$ is a {\it generic} generalized conformal structure on the manifold $M$ \cite[Def. 1.2]{BZ17} if and only if $\epsilon'(s)\neq 0$. Under the assumption $\epsilon'(s)\neq 0$,
$$
N\to \mathrm{Sym}^{+}(TM), \quad y=\Gamma(x)\cdot s \mapsto \epsilon^{2}(s)g_{1}\mid_{T_{x}M}
$$
is a {\it classical} conformal structure on $M$.
There is a natural choice for $g^{\sigma}$ when the metrics $g_{s}$ are conformally related and the metric $g_{1}$ is Einstein with 
$\mathrm{Ric}^{g_{1}}=2\lambda m g_{1}$ ($m\geq 3$). Namely, the Ricci flat metric \cite[Sec. 7]{F-G}
$$g^{\sigma}=ds\otimes d(\rho s)+d(\rho s)\otimes ds +(1+\lambda\rho)^2s^{2}g_{1}.$$
}
\end{remark}

As a consequence of Remark \ref{warped}, the formulas for the Levi-Civita connection and the curvature of warped product metrics \cite[Chap. 7]{One83} do not work for the metrics $g^{\sigma}$. The following Lemmas provide such formulas. Let $\pi_{B}$ and $\pi_{M}$ be the projections of ${\bf M}=B \times M$ onto $B$ and $M$, respectively. The vector fields on $B$ or $M$ can be lifted to ${\bf M}$. The sets of such lifts will be denoted by $\mathfrak{L}(B)$ and $\mathfrak{L}(M)$, respectively. We simplify the notation by writing $\sigma$ for $\sigma\circ \pi_{1}$, where $\pi_{1}:{\bf M}\to (-\varepsilon, \varepsilon)$ is the projection on the first factor. 
Following the terminology for warped products,  the vectors tangent to $B$ are called horizontal and vectors tangent to $M$ vertical. We write $\mathbf{nor}$ and $\mathbf{tan}$ for the projections from $\mathfrak{X}({\bf M})$ to $\mathfrak{L}(B)$ and $\mathfrak{L}(M)$, respectively. 
The vector fields 
\begin{equation}\label{21}
T=\frac{1}{\sqrt{2}}\Big[(1+\rho/s^2)\partial_{\rho}-(1/s)\partial_{s}\Big],\,\, E=\frac{1}{\sqrt{2}}\Big[(1-\rho/s^2)\partial_{\rho}+(1/s)\partial_{s}\Big] \in \mathfrak{X}(B)\subset \mathfrak{L}(B)
\end{equation}
 give an orthnormal basis of $g_{B}$ with $T$ timelike. Thus, the Lorentzian manifold $({\bf M}, g^{\sigma})$ is timelike orientable.

\begin{lemma}\label{LC}
On $({\bf M}, g^{\sigma})$, for $V,W\in \mathfrak{L}(M)$, we have for the Levi-Civita connection $\nabla^{\sigma}$
\begin{enumerate}
\item $
\nabla^{\sigma}_{\partial_{\rho}}\partial_{\rho}=0, \,\,\nabla^{\sigma}_{\partial_{s}}\partial_{\rho}=\frac{1}{s}\partial_{\rho},\,\, \nabla^{\sigma}_{V}\partial_{\rho}=\frac{\sigma'}{\sigma}V \,\,\textrm{and}\,\,\,\nabla^{\sigma}_{\partial_{s}}\partial_{s}=0.$ In particular, the expansion function of the lightlike hypersurface  $M\times \R_{>0}$ in ${\bf M}$ at $\rho=0$ satisfies 
$\lambda=1$.

\item $\mathbf{tan}(\nabla^{\sigma}_{V}W\vert_{(\rho,x,s)})\in \mathfrak{L}(M)$ is the lift of $\nabla^{s}_{V}W$ where $\nabla^s$ is the Levi-Civita connection of $g_{s}=(\Gamma_{s})^{*}(h)$ and 
$$\mathbf{nor}(\nabla^{\sigma}_{V}W\vert_{(\rho,x,s)})=\frac{\sigma'}{s\sigma}g^{\sigma}(V,W )\Big(-\partial_{s}+ \frac{2\rho}{s}\partial_{\rho}\Big)-\frac{1}{s}g^{\sigma}\Big(\nabla^{\sigma}_{\partial_{s}}V,W\Big)\partial_{\rho}.
$$
\end{enumerate}
\end{lemma}
\begin{proof} 
 Taking into account that $\partial_{\rho}, \partial_{s}\in \mathfrak{L}(B)$ and $V\in \mathfrak{L}(M)$, we have $[\partial_{\rho},V]=[\partial_{s},V]=0$ and, of course $[\partial_{\rho}, \partial_{s}]=0$ and $[V,W]\in \mathfrak{L}(M)$. Now, the Koszul formula gives $g^{\sigma}(\nabla^{\sigma}_{\partial_{\rho}}\partial_\rho, F)=g^{\sigma}(\nabla^{\sigma}_{\partial_{s}}\partial_s, F)=0$ for all $F\in \mathfrak{X}({\bf M})$ and 
$$g^{\sigma}(\nabla^{\sigma}_{\partial_{s}}\partial_\rho, V)=0,\,\, g^{\sigma}(\nabla^{\sigma}_{\partial_{s}}\partial_\rho, \partial_{s})=1, \,\,
g^{\sigma}(\nabla^{\sigma}_{V}\partial_\rho, \partial_s)=0,\,\, g^{\sigma}(\nabla^{\sigma}_{V}\partial_\rho, W)=\frac{\sigma'}{\sigma}g^{\sigma}(V,W).$$
Thus, we obtain the required formulas in item $1)$.

Since $V$ and $W$ are tangent to $\{\rho\}\times M \times \{s\}$, we know that $\mathbf{tan}(\nabla^{\sigma}_{V}W\mid_{(\rho,x,s)})$ is the covariant derivative applied to $V$ and $W$ restricted on $\{\rho\}\times M \times \{s\}$ (see for instance \cite[Chap. 4]{One83}). The metric on $\{\rho\}\times M \times \{s\}$ is $\sigma^{2}(\rho)g_{s}$ and hence, it is homothetic to $g_{s}$ and the homotheties preserves the Levi-Civita connection.
Finally, we proceed as follows
$$
\mathbf{nor}(\nabla^{\sigma}_{V}W)=-g^{\sigma}(\nabla^{\sigma}_{V}W, T)T+ g^{\sigma}(\nabla^{\sigma}_{V}W, E)E=
g^{\sigma}(\nabla^{\sigma}_{V}T, W)T- g^{\sigma}(\nabla^{\sigma}_{V}E,W)E.
$$
Substituting (\ref{21}) into this formula and using item $2)$, we conclude the proof. 
\end{proof}

\begin{remark}
{\rm When the metric $g^{\sigma}$ is a warped product, we know from Remark \ref{warped} that 
$g^{\sigma}=g_{B}+\sigma^{2}(\rho)\epsilon^{2}(s)g_{1}$ and hence $\nabla^{\sigma}_{\partial_{s}}V=(\epsilon'(s)/\epsilon(s)) V$.
In this case, Lemma \ref{LC} reduces to the formulas for the Levi-Civita connection of such warped product metrics \cite[Prop. 7.35]{One83}. 
On the contrary to the warped product metrics, the spacelike fibers $\{\rho \}\times M\times \{s \}\subset{\bf M}$ are not totally umbilical, in general. The mean curvature vector field of every such spacelike fibers can be computed from Lemma \ref{LC} as follows
$$\overrightarrow{{\bf H}}(\rho,x,s)=\frac{\sigma'}{s\sigma}\Big(-\partial_{s}+ \frac{2\rho}{s}\partial_{\rho}\Big)-\frac{1}{ms^{2}}\left(s\,\mathrm{div}^{\sigma}(\partial_{s})-1\right)\,\partial_{\rho}.
$$

}
\end{remark}

\begin{lemma}\label{Rs}
On $({\bf M}, g^{\sigma})$, for $V,W\in \mathfrak{L}(M)$, we have for the curvature tensor $R^{\sigma}$
\begin{enumerate}
\item $R^{\sigma}(\partial_{s},\partial_{\rho})\partial_{\rho}=0,\,\, R^{\sigma}(V,\partial_{\rho})\partial_{\rho}=-(\sigma''/\sigma)V,\,\, R^{\sigma}(\partial_{\rho},\partial_{s})\partial_{s}=-(2/s^{2})\partial_{\rho}$.

\item $R^{\sigma}(V, \partial_{s})\partial_{s}= - \nabla^{\sigma}_{\partial_{s}}(\nabla^{\sigma}_{V}\partial_{s})$,\,\, $R^{\sigma}(V, \partial_{\rho})\partial_{s}=\frac{\sigma '}{s\sigma}V- \nabla^{\sigma}_{\partial_{\rho}}(\nabla^{\sigma}_{V}\partial_{s})$.

\end{enumerate}
In particular, we get $\mathrm{Ric}^{\sigma}(\partial_{\rho}, \partial_{\rho})=-m\frac{\sigma''}{\sigma}$.
\end{lemma}

\begin{remark}\label{11122}
{\rm
In order to link the ideas of this section with the Fefferman-Graham construction \cite{F-G},
we would like to recall the notion of ambient manifolds for Riemannian conformal structures as appeared in \cite{CG03}. Let $(M, [g])$ be a conformal Riemannian structure on a $(n\ge2)$-dimensional manifold $M$.
Let ${\bf N}\subset \mathrm{Sym}^{+}(M)$ be the fiber bundle of scales of the conformal structure and $\pi :{\bf N}\to M^{m}$ the projection. Every section of $\pi$ provides a  representative $g\in [g]$.
On ${\bf N}$ there are a free $\R_{>0}$-action (called dilations) defined by $\delta(s)(u)= s^{2}u$ for $u \in {\bf N}$ and a tautological symmetric $2$-tensor $h$ given by $h(\xi, \eta)=u(T_{u}\pi \cdot \xi, T_{u}\pi \cdot \eta)$ for  $\xi, \eta \in T_{u}{\bf N}$. The tautological tensor $h$ is a lightlike metric on ${\bf N}$.
An ambient manifold is a $(m+2)$-dimensional manifold ${\bf M}$ endowed with
a free $\R_{>0}$-action and a $\R_{>0}$-equivariant embedding $i:{\bf N} \to {\bf M}$. 
If ${\bf M}$ is an ambient manifold, then an ambient metric is a Lorentzian metric $g_{_{L}}$ on ${\bf M}$ such that the following conditions hold.
\begin{enumerate}
\item The metric is homogeneous of degree $2$ with respect to the $\R_{>0}$-action. That is, if $X\in \mathfrak{X}({\bf M})$ denotes the fundamental vector field generating the $\R_{>0}$-action, then we have $\mathcal{L}_{X}g_{_{L}}=2g_{_{L}}$.
\item For all $u=g_{x}\in {\bf N}$ and $\xi, \eta \in T_{u}{\bf N}$, we have $i^{*}(g_{_{L}})(\xi, \eta)=g_{x}(T_{u}\pi \cdot\xi, T_{u}\pi \cdot\eta)$. That is, $i^{*}(g_{_{L}})=h$.
\end{enumerate}
The original definition of an ambient metric in \cite{F-G} adds certain conditions of Ricci-flatness for the metric $g_{_{L}}$.
This approach to the study of conformal geometry is in the roots of this paper. In fact, we can read this construction in terms of lightlike manifolds as follows. To start with, we focus on the lightlike manifold $({\bf N}, h , Z)$ where $Z$ is the fundamental vector field generating the dilations. The vector field $Z$ is complete and the action (\ref{301}) is free and proper (recall that $\pi:{\bf N}\to M$ is a $\R_{>0}$-principal fiber bundle). 
Now, the {\it surface of the light source} is the conformal Riemannian manifold $M={\bf N}/\R_{>0}$. The ambient manifold ${\bf M}$ is a Lorentzian manifold and $i: {\bf N}\to {\bf M}$ realizes $({\bf N}, h , Z)$ as a lightlike hypersurface. The vector field $X\in \mathfrak{X}({\bf M})$ satisfies $g_{_{L}}(X,X)\circ i=i^{*}(h)(Z,Z)=0$ and $i:({\bf N}, h, Z)\to ({\bf M}, g_{L})$ is causally oriented with respect to the timelike orientation induced by $X$ on $({\bf M}, g_{L})$. Morever, the metric $g_{_{L}}$ is homogeneous of degree $2$ implies that
$\overline{\nabla}^{g_{_{L}}}Z=\mathrm{Id}$. Hence, Theorem \ref{main1} can be applied to $i:({\bf N}, h, Z)\to ({\bf M}, g_{L})$ and shows that there is a Cartan connection $\omega$ on the fiber principal bundle $Q\subset L({\bf N})$ of all admissible linear frames of ${\bf N}$ which induces the original tautological lightlike metric $h$ and the vector field $Z$. 
}
\end{remark}

\begin{remark}
{\rm Assume $m\geq 3$ and let $(p:\mathcal{P}\to M, \omega^{[g]})$ be the canonical Cartan geometry of type $(G,P)$ corresponding to the conformal structure $(M,[g])$ (see for instance \cite[Theor. 1.6.7]{CS09}. Then, the quotient manifold $\mathcal{P}/H$ can be seem as the fiber bundle of scales ${\bf N}$ of the conformal structure \cite[Sec. 1.6.3]{CS09}.
}
\end{remark}

We now return to our setting. Let $(N,h,Z)$ be a generic lightlike manifold and assume $Z$ complete with the $\R_{>0}$-action 
given in (\ref{301}) free and proper. Consider the orbit space $M=N/\R_{>0}$ with projection $\pi:N \to M$ and fix a section $\Gamma: M \to N$. Then,
the diffeomorphism (\ref{diff}) shows the lightlike manifold $(N,h,Z)$ as the lightlike hypersurface $(M\times \R_{>0}, q, s\partial_{s})$ embedded at $\rho=0$ in $({\bf M}, g^{\sigma})$. Explicitly, we have
\begin{equation}\label{51}
\psi:(N,h,Z)\to ({\bf M}, g^{\sigma}), \,\, y\mapsto (0, \pi(y),s),\,\, \mathrm{where } \,\,y=\mathrm{Fl}^{Z}_{\log s}(\pi(y)).
\end{equation}
The lightlike hypersurface (\ref{51}) is causally oriented with respect to the vector field $-T$ in $(\ref{21})$ and, from  Lemma \ref{LC}, we know that expansion function $\lambda=1$. 
The generic condition implies the null-Weingarten map $[\overline{\nabla}^{\sigma}Z]$ is a vector fiber bundle isomorphism (\ref{240120}).
Moreover, Lemma \ref{Rs} gives that $\mathrm{Ric}^{\sigma}(\partial_{\rho}, \partial_{\rho})=0$ reduces the family of metrics $g^{\sigma}$
to the one-parametric subfamily
$$g^{c}=ds\otimes d(\rho s)+d(\rho s)\otimes ds + (1+c\rho)^{2}g_{s}, \quad c\in \R.$$ 
Let $\gamma^{c}\in \Omega^{1}(O^{+}({\bf M}), \mathfrak{g})$ be the (Ehresmann) principal connection form corresponding to the Levi-Civita connection $\nabla^c$ of $g^{c}$ and let $Q$ be the principal fiber bundle
of admissible linear frames of $N$. We are in position to apply Theorem \ref{main1} for lightlike hypersurfaces as in (\ref{51}).

\begin{theorem}\label{61}
Let  $(N^{m+1}, h, Z)$ be a generic lightlike manifold with $Z$ complete. Assume the $\R_{>0}$-action given by the flow of $Z$ (\ref{301}) is free and proper. Then, for every spacelike section $\Gamma:N/\R_{>0}:=M\to N$ and $c\in \R$, the pull-back
$\omega^{c}:=\psi^{*}(\gamma^{c})\in \Omega(Q, \mathfrak{g})$ is a Cartan connection. The lightlike metric $h^{c}$ and the vector field $Z^{c}$ deduced from $\omega^{c}$ are
$h^{c}_{[\overline{\nabla}^{c}Z]^{-1}}= h$ and  $Z^{c}=Z$. 
\end{theorem}

\begin{example}
{\rm Following \cite{BZ17}, a lightlike manifold $(N,h)$ is said to be simple if 
\begin{enumerate}
\item  There is a submersion $\pi:N^{m+1}\to M^{m}$ 
such that the connected components of the fibers $\pi^{-1}(x)$ are the leaves of the lightlike foliation. 

\item For every $x\in M$, the set of all inner products $g_{y}$ on $T_{x}M$ obtained from the projections $T_{y}N\to T_{x}M$ is a regular generalized conformal structure. That is, the map 
$$
N \to \mathrm{Sym}^{+}(M),\quad y \mapsto g_{y}
$$ is a submanifold tranverse to the fibers of the projection $\mathrm{Sym}^{+}(M)\to M$. 
\end{enumerate}
Let $(N,h, Z)$ be a simple and generic lightlike manifold. For every section $\Gamma$ of $\pi:N\to M$, there is a one parameter family of Cartan connections on the principal fiber bundle $Q$ of all admissible linear frames of $(N,h,Z)$. 
}
\end{example}

\begin{example}
{\rm When the metrics $g^{\sigma}$ are warped products (see Remark \ref{warped})
\begin{equation}
g^{\sigma}\mid_{(\rho, x,s)}=ds\otimes d(\rho s)+d(\rho s)\otimes ds +(1+c\rho)^{2}\epsilon^{2}(s)g_{1},
\end{equation}
the lightlike hypersurface (\ref{51}) is totally umbilical with $B_{Z}=\frac{\epsilon'}{\epsilon}\overline{h}$, and the generic condition holds if and only if $\epsilon'(s) \neq 0$ for $s>0$.
}
\end{example}

\begin{example}
{\rm Let $(M,g_{0})$ be a Riemannian manifold. The Ricci flow is the PDE 
$$
\frac{\partial }{\partial t}g(t)=-2 \mathrm{Ric}(g(t)), \quad g(0)=g_{0},
$$
where $g(t)$ is a one-parametric family of Riemannian metrics on $M$ and $\mathrm{Ric}(g(t))$ are the corresponding Ricci tensors. There exists $T>0$ such that the solution $g(t)$ of the Ricci flow exists and is smooth and unique on $[0,T)$.  The lightlike manifold
\begin{equation}\label{81}
(N:=M\times [0,T), h= g(t)\oplus 0, Z= \partial_{t})
\end{equation}
is generic if and only if the Ricci tensors $\mathrm{Ric}(g(t))$ are non degenerate.}
\end{example}



\end{document}